\documentclass[12pt,a4paper,final,reqno]{amsart}

\newcommand{\includefigure}[2]{\includegraphics[#1]{#2}}

\usepackage{amsmath}
\usepackage{amsthm}
\usepackage{amssymb}
\usepackage{graphicx}
\usepackage{mathrsfs}
\usepackage[english]{babel}
\usepackage[latin1]{inputenc}
\usepackage[T1]{fontenc}
\usepackage{amsrefs}
\usepackage[usenames,dvipsnames]{xcolor}
\usepackage{enumitem}
\usepackage{xcolor}
\usepackage{bbm}
\usepackage{multirow}

\numberwithin{equation}{section}

\newtheorem{theorem}{Theorem}[section]
\newtheorem{proposition}[theorem]{Proposition}

\newtheorem{corollary}[theorem]{Corollary}
\newtheorem{lemma}[theorem]{Lemma}

{\theoremstyle{definition}
\newtheorem{remark}[theorem]{Remark}

\newtheorem{defn}[theorem]{Definition}
}

\newtheorem{Rule}{Rule}
\newtheorem{PolyRule}{Rule}

\newcommand{\Gcal}{\mathcal{G}}

\newcommand{\Nn}{{\mathbb{N}}}

\newcommand{\Rr}{{\mathbb{R}}}

\def\diag{\operatorname{diag}}

\newcommand{\Mat}{{\rm Mat}}

\newtheorem*{theorem*}{Theorem}

\newcommand{\abs}[1]{\left| #1\right|}

\newcommand{\comment}[1]{}

\newcommand{\Bscr}{\mathscr{B}}

\newcommand{\Escr}{\mathscr{E}}

\newcommand{\Fscr}{\mathscr{F}}

\newcommand{\Rscr}{\mathscr{R}}

\newcommand{\Vscr}{\mathscr{V}}

\newcommand{\nund}{{\underline{n}}}
\newcommand{\mund}{{\underline{m}}}
\newcommand{\rund}{{\underline{r}}}

\newcommand{\inter}{{\rm int}}

\newcommand{\bbEdge}{\bullet\!\!-\!\!\bullet}
\newcommand{\blV}{\bullet}

\newcommand{\Nuc}{{\rm Ker}}
\newcommand{\rank}{{\rm rank}}
\newcommand{\attr}[1]{\Lambda_{{#1}}}

\newcommand{\xred}[1]{\check{x}^{{#1}}}

\begin{document}
\selectlanguage{english}
\title[Dissipative Polymatrix Replicators]{Conservative and Dissipative Polymatrix Replicators}
\subjclass[2010]{91A22, 37B25, 70G45}
\keywords{evolutionary game theory, replicator equation, stably dissipative, Hamiltonian systems, Lotka-Volterra systems}

\author[Alishah]{Hassan Najafi Alishah}
\address{ Departamento de Matem\'atica, Instituto de Ci\^encias Exatas\\
Universidade Federal de Minas Gerais \\
31123-970 Belo Horizonte\\
MG - Brazil }
\email{halishah@mat.ufmg.br}

\author[Duarte]{Pedro Duarte}
\address{Departamento de Matem\'atica and CMAF \\
Faculdade de Ci\^encias\\
Universidade de Lisboa\\
Campo Grande, Edificio C6, Piso 2\\
1749-016 Lisboa, Portugal 
}
\email{pmduarte@fc.ul.pt}

\author[Peixe]{ Telmo Peixe}
\address{Departamento de Matem\'atica and CMAF \\
Faculdade de Ci\^encias\\
Universidade de Lisboa\\
Campo Grande, Edificio C6, Piso 2\\
1749-016 Lisboa, Portugal 
}
\email{tjpeixe@fc.ul.pt}

\begin{abstract}
In this paper we address a class of replicator dynamics, referred as polymatrix replicators,
that contains well known classes of evolutionary game dynamics, such as
the symmetric and asymmetric (or bimatrix) replicator equations, and some   replicator equations for $n$-person games. Polymatrix replicators form a simple class of
algebraic o.d.e.'s on prisms (products of simplexes), which describe the evolution of strategical behaviours within a po\-pu\-lation stratified in  $n\geq 1$ social groups. 

In the 80's Raymond Redheffer et al. developed a theory on the  class of stably dissipative Lotka-Volterra systems.
This theory is built around a reduction algorithm that ``infers'' the localization of the system' s attractor in some affine subspace.
It was later proven that the dynamics on the attractor of such systems is always embeddable in a Hamiltonian Lotka-Volterra system.  

In this paper we  extend these results to polymatrix replicators.
\end{abstract}

\maketitle


\section{Introduction}
\label{intro}
Lotka-Volterra (LV) systems were  introduced independently by 
Alfred Lotka~\cite{Lot1958} and Vito Volterra \cite{Volt1990} to model the evolution of biological and chemical ecosystems. The phase space of a 
Lotka-Volterra system is the non-compact polytope
$\Rr^n_+=\{x\in\Rr^n: x_i\geq0,\; i=1,\ldots,n\}$, where a point in $\Rr^n$ represents a state of the ecosystem.
The   LV systems are defined by the following o.d.e. 
$$ \frac{dx_i}{dt} = x_i\, f_i(x), \quad  i=1,\ldots , n , $$
where usually  the so called {\em fitness functions} $f_i(x)$ are considered to be affine, i.e., of the form
$$ f_i(x)= r_i + \sum_{j=1}^n a_{ij}\,x_j , $$
where  $A=(a_{ij})\in \Mat_{n\times n}(\Rr)$ is called the system's interaction matrix.

In general, the dynamics of LV systems can be arbitrarily rich, 
as was first observed by S. Smale~\cite{Sma1976} who proved that any finite dimensional compact flow can be embedded in a LV system with non linear fitness functions. Later, using a class of embeddings studied by L. Brenig~\cite{Brenig1988}, B. Hern{\'a}ndez-Bermejo and V. Fair{\'e}n ~\cite{BF1997}, it was  proven (see~\cite[Theorems 1 and 2]{BF1997}) that any LV system with polynomial fitness functions can be  embedded in a LV system with affine fitness functions. 
Combining this  with Smale's result, we infer that any finite dimensional compact  flow can be, up to a small perturbation, embedded in a LV system with affine fitness functions.  These facts emphasize the difficulty of studying the  general dynamics of LV systems.

In spite of  these difficulties, many dynamical consequences have been driven from information on the fitness data $f_i(x)$  for some special classes of LV systems.
Two such classes are the {\em cooperative} and {\em competitive} 
LV systems, corresponding to fitness functions satisfying $\frac{\partial f_i}{\partial x_j}\geq 0$ and  $\frac{\partial f_i}{\partial x_j}\leq 0$, respectively, for all $i,j$. Curiously, the fact that
Smale's embedding takes place in a competitive LV system influenced the development of the theory of cooperative and  competitive LV systems initiated by M. Hirsch~\cites{Hir1982,Hir1985,Hir1988}.

In his pioneering work Volterra~\cite{Volt1990} studies dissipative LV systems as  generalizations of the classical predator-prey model. 
A LV system with interaction matrix $A=(a_{ij})$ is called {\em dissipative}, resp. {\em conservative}, if there are constants $d_i>0$ such that
the quadratic form $Q(x)=\sum_{i,j=1}^n a_{ij} d_j x_i x_j$ is negative semi-definite, resp. zero.
Note that the meaning of the term {\em dissipative} is not strict
because dissipative LV o.d.e.s include conservative LV systems.
In addition we remark that conservative LV models are in some sense Hamiltonian systems, a fact that was well known and  explored by Volterra.

Given a LV system with interaction matrix $A=(a_{ij})$,
we define its interaction graph $G(A)$ to be the undirected graph with vertex set $V=\{1,\ldots, n\}$ that includes an edge
connecting $i$ to $j$ whenever  $a_{ij}\neq 0$ or $a_{ji}\neq 0$.
The LV system and its matrix $A$ are called {\em stably dissipative} if $\sum_{i,j=1}^n \tilde a_{ij} x_i x_j\leq 0$ for all $x\in\Rr^n$ and every small enough perturbation
$\tilde A=(\tilde a_{ij})$ of $A$ such that $G({\tilde{A}})=G(A)$.
This notion of stably dissipativeness is due to Redheffer \textit{et al.} whom in a series of papers \cites{RZ1981,RZ1982,RW1984,Red1985,Red1989} studied 
this class of models under the name of {\em stably admissible} systems.

Assuming the system admits an interior equilibrium $q\in {\rm int}(\Rr_+^n)$,
Redheffer \textit{et al.} describe a simple reduction algorithm, running on the graph $G(A)$,
that `deduces' the minimal affine subspace of the form $\cap_{i\in I} \{ x\in\Rr^n_+\,:\, x_i=q_i\}$ that contains the attractor of every stably dissipative LV system with interaction graph $G(A)$.

Under the scope of this theory, Oliva \textit{et al.}~\cite{DFO1998}  
have proven that the dynamics on the attractor of a stably dissipative LV system is always described by a conservative (Hamiltonian) LV system.

 The replicator equation, which is now central to Evolutionary Game Theory (EGT), was  introduced by  P. Taylor and L. Jonker~\cite{TJ1978}. It models the time evolution of the probability distribution of strategical behaviors within a biological population.
Given a {\em payoff matrix} $A\in \Mat_{n\times n}(\Rr)$, the replicator equation refers to the following o.d.e.
$$  x_i' = x_i\, 
\left( (A x)_i - x^t\, A\, x \right),\quad i=1,\ldots, n   $$
on the simplex
$\Delta^{n-1}=\{x\in\Rr_+^n: \sum_{j=1}^n x_j= 1 \}$. 
This equation says that the logarithmic growth of the usage frequency of each behavioural strategy is directly proportional to how well that strategy fares within the population.

Another important class of models in EGT,  that includes the {\em Battle of sexes}, is the {\em bimatrix replicator} equation.
In this model the population is divided in two groups, 
e.g. males and females, and all interactions involve individuals of different groups.
Given two payoff matrices
$A\in \Mat_{n\times m}(\Rr)$ and
$B\in \Mat_{m\times n}(\Rr)$, for the strategies in each group, the bimatrix replicator refers to the   o.d.e. 
$$ \left\{
\begin{array}{lll}
x_i' &= x_i\, 
\left( (A y)_i - x^t\, A\, y \right) & i=1,\ldots, n\\
y_j'  &= y_j\, 
\left( (B x)_j - y^t\, B\, x \right) & j=1,\ldots, m
\end{array}
\right.  $$
on the product of simplices $\Delta^{n-1}\times\Delta^{m-1}$. 
It describes the time evolution of the strategy usage frequencies  in each  group. These systems were first studied in ~\cite{SS1981} and ~\cite{SSHW1981}.

We now introduce the {\em polymatrix replicator} equation studied in ~\cite{AD2015}.
Consider a population is divided in $p\in\Nn$ groups,
$\alpha=1,\ldots, p$, each  with $n_\alpha\in\Nn$  behavioral strategies, in a total of $n=\sum_{\alpha=1}^p n_\alpha$ strategies, numbered from $1$ to $n$. The system is described by a single payoff matrix $A\in\Mat_{n\times n}(\Rr)$, which can be decomposed in $p^2$ blocks
$A^{\alpha, \beta}
\in \Mat_{n_\alpha\times n_\beta}(\Rr)$ with the payoffs corresponding to interactions between strategies in group $\alpha$ with strategies in group $\beta$.
Let us abusively write $i\in\alpha$ to express that $i$ is a strategy of the group $\alpha$. With this notation the polymatrix replicator refers to the following o.d.e.
$$ x_i' = x_i\,\left( (A\,x)_i - \sum_{j\in\alpha} x_j\,(A\,x)_j \right), \, i\in \alpha, \alpha\in \{1,\ldots, p\}. $$
on the product  of simplexes
$\Delta^{n_1-1}\times \ldots \times \Delta^{n_p-1}$.
Notice that interactions between individuals of any two groups (including the same)  are allowed.
Notice also that this equation implies that
 competition takes place inside the groups, i.e., the relative success of each strategy is evaluated within the corresponding group.

This class of evolutionary systems includes both the replicator equation (when $p=1$) and the bimatrix replicator equation (when $p=2$ and  $A^{1,1}=0=A^{2,2}$). It  also includes the replicator equation for $n$-person games (when $A^{\alpha,\alpha}=0$ for all $\alpha=1,\ldots, p$). This last subclass of polymatrix replicator equations 
specializes more general replicator equations for $n$-person games with multi-linear payoffs that were first formulated by  Palm~\cite{Palm} and studied by Ritzberger, Weibull~\cite{KRJW}, Plank~\cite{Plank3} among others.

In this paper we define the class of admissible polymatrix replicators (the analogue of stably dissipative for LV systems), and introduce a reduction algorithm  similar to the one of Redheffer  that `deduces' the   constraints on the localization of the attractor. We also generalize the mentioned theorem in~\cite{DFO1998} to polymatrix replicator systems.

This paper is organized as follows.
In Section~\ref{poly} we introduce the notion of polymatrix game as well as its associated polymatrix replicator system (o.d.e.), proving some elementary facts about this class of models.
In Section~\ref{lvs} we recall some known results of Redheffer {\it et al.} reduction theory for stably dissipative LV systems.
In sections~\ref{hamiltonian} and~\ref{dissipative}, we
define, respectively, the classes of conservative and dissipative polymatrix replicators,
and study their properties.
In particular, we extend to polymatrix replicators the concept
of stably dissipativeness of
Redheffer {\em et al.}.
We generalize to this context the mentioned theorem in~\cite{DFO1998}
about the Hamiltonian nature of the limit dynamics of a ``stably dissipative'' system.
Finally, in Section~\ref{applications}  we illustrate our results with a simple example.


\section{Polymatrix Replicators}
\label{poly}

\begin{defn}
A polymatrix game is an ordered  pair $(\nund,A)$ where 
$\nund=(n_1,\ldots, n_p)$ is a list of positive integers, called the game type, and $A\in \Mat_{n\times n}(\Rr)$ a square matrix of dimension  $n=n_1+\ldots + n_p$.
\end{defn}

This formal definition has the following interpretation.

Consider a population divided in $p$ groups, labeled by an integer
$\alpha$ ranging from $1$ to $p$. Individuals of each group 
$\alpha=1,\ldots, p$ have exactly $n_\alpha$ strategies to interact with
other members of the population. The strategies of a group $\alpha$ are labeled by positive integers $j$ in the range 
$$ n_1+\ldots + n_{\alpha-1} <j \leq n_1+\ldots + n_{\alpha}\;.$$
We will write $j\in \alpha$ to mean that $j$ is a strategy of the group $\alpha$.
Hence the strategies of all groups are labeled by the  integers $j=1,\ldots, n$.

The matrix $A$ is the {\em payoff matrix}.
Given strategies $i\in\alpha$ and $j\in\beta$, in the groups $\alpha$ and $\beta$ respectively, the entry $a_{i j}$
represents an average payoff for an individual using the first strategy in some interaction with an individual using the second. Thus, the payoff matrix $A$ can be decomposed
into $n_\alpha\times n_\beta$ block matrices  $A^{\alpha,\beta}$, with entries $a_{ij}$, $i\in\alpha$ and $j\in \beta$,
where $\alpha$ and $\beta$ range from $1$ to $p$.

\begin{defn}\label{equivalent matrices defn}
Two polymatrix games $(\nund,A)$ and $(\nund,B)$ with the same type are said to be equivalent, and we write $(\nund,A)\sim (\nund,B)$, when
for $\alpha,\beta=1,\ldots, p$, all the rows of the block matrix  $A^{\alpha \beta}-B^{\alpha \beta}$ are equal.
\end{defn}

The {\em state} of the population is described by a point $x=(x^ \alpha)_\alpha$ in the {\em prism}
$$ \Gamma_{\nund} :=\Delta^{n_1-1}\times \ldots \times \Delta^{n_p-1}\subset \Rr^n \;, $$
where $\Delta^{n_\alpha-1}=\{x\in\Rr^{n_\alpha}: \sum\limits_{i=1}^{n_\alpha}x_i=1\}$,
$x^\alpha=(x_j)_{j\in\alpha}$ 
and the entry $x_j$ represents the  usage frequency of strategy $j$ within the group $\alpha$.
The prism  $\Gamma_{\nund}$ is a $(n-p)$-dimensional simple polytope whose affine support
is the $(n-p)$-dimensional space $E^ {n-p}\subset\Rr^ n$ defined by the $p$ equations
$$ \sum_{i\in\alpha} x_i  = 1,\quad 1\leq \alpha\leq p\;.$$

\begin{defn}
A polymatrix game $(\nund,A)$ determines the following o.d.e.
on the prism $\Gamma_{\nund}$
\begin{equation}\label{ode:pmg}
\frac{ d x_i}{dt} = x_i\,\left( (A\,x)_i - \sum_{\beta=1}^p (x^{\alpha})^ T A^{\alpha,\beta} x^{\beta} \right), \quad \forall\; i\in \alpha,\; 1\leq \alpha\leq  p \;,
\end{equation}
called a polymatrix replicator system.
\end{defn}

This equation says that the logarithmic growth rate of each frequency $x_i$ is the difference between its
payoff $(A x)_i=\sum_{j=1}^n a_{ij} x_j$  and the average payoff of all strategies in the group $\alpha$.
The flow $\phi_{\nund,A}^t$ of this equation leaves the prism $\Gamma_{\nund}$ invariant.
Hence, by compactness of $\Gamma_{\nund}$, this flow is complete.
The underlying vector field on $\Gamma_{\nund}$ will be denoted by $X_{A,\nund}$.

In the case $p=1$, we have $\Gamma_{\nund}=\Delta^{n-1}$ and 
~\eqref{ode:pmg} is the usual replicator equation associated to the payoff matrix $A$.

When $p=2$, and  $A^{11}=A^{22}=0$,  
$\Gamma_{\nund} =\Delta^ {n_1-1}\times \Delta^{n_2-1}$  and ~\eqref{ode:pmg} becomes 
the  bimatrix replicator equation associated to the pair of payoff matrices $(A^{12},A^ {21})$.

The polytope $\Gamma_{\nund}$
is parallel to the affine subspace
\begin{equation}\label{def Hnund}
H_{\nund}:=\left\{\, x\in\Rr^n\,\colon\,
\sum_{j\in\alpha} x_j=0,\; \text{ for }\; \alpha=1,\ldots, p\, \right\}\;.  
\end{equation}

\newcommand{\one}{{\mathbbm{1}}}

For each $\alpha=1,\ldots, p$, we denote by $\pi_\alpha:\Rr^n\to \Rr^n$ the projection
$$ x \mapsto y,\quad y_i:= \left\{
\begin{array}{cll}
x_i &\text{ if } & i\in \alpha \\
0 &\text{ if } & i\notin \alpha \\
\end{array}\right. \;. $$
We also define $\one:= (1,\dots ,1)\in\Rr^n\,.$

\begin{lemma} \label{lemma equivalence}
Given a matrix $C\in\Mat_{n\times n}(\Rr)$, the following statements are equivalent:
\begin{enumerate}
\item[(a)] $C^{\alpha \beta}$ has equal rows, for all 
$\alpha,\beta\in\{1,\ldots, p\}$,
\item[(b)] $C x\in H_\nund^\perp$, for all $x\in\Rr^n$.
\end{enumerate}
Moreover, if any of these conditions holds then
$X_{\nund,C}=0$ on $\Gamma_\nund$.
\end{lemma}

\begin{proof}
Assume (a).
Since $H_\nund^\perp$ is spanned by the vectors
$\pi_\alpha(\one)$ with $\alpha=1,\ldots, p$,  
we have $v\in H_\nund^\perp$ iff $v_i=v_j$ for all $i,j\in\alpha$.
Because all rows of $C$  in the  group $\alpha$ are equal, we have $(C x)_i=(C x)_j$ for all $i,j\in\alpha$.
Hence item (b) follows.

Next assume (b).
For all $i\in\alpha$, with $\alpha\in\{1,\ldots, p\}$,
$C e_i\in H_\nund^\perp$, which implies that
$c_{i,k}=c_{j,k}$ for all $j\in\alpha$. This proves (a).

If  (a) holds, then for any $\alpha\in\{1,\dots,p\}$,  $i,j\in\alpha$ and $k=1,\ldots, n$,
we have   $c_{ik}=c_{jk}$. Hence for any $x\in\Gamma_{\nund}$,  and $i,j\in\alpha$ with $\alpha\in\{1,\ldots, p\}$,
$(C\,x)_i=(C\, x)_j$, which implies that
$X_{\nund,C}=0$ on $\Gamma_\nund$.
\end{proof}

\begin{proposition}\label{equivalent matrices prop}
Given two polymatrix games  $(\nund,A)$ and  $(\nund,B)$ with the same type $\nund$, if $(\nund,A)\sim (\nund,B)$ then 
  $X_{\nund,A}=X_{\nund,B}$ on $\Gamma_{\nund}$.

\end{proposition}

\begin{proof}
Follows from Lemma~\ref{lemma equivalence} and the linearity of the correspondence
$A\mapsto X_{\nund,A}$.
\end{proof}

We have the following obvious characterization of interior equilibria.

\begin{proposition}
\label{prop int equilibria}
Given a polymatrix game  $(\nund,A)$, a point
$q\in {\rm int}(\Gamma_\nund)$ is an equilibrium of $X_{\nund,A}$ \, if and only if \,
$(A\,q)_i=(A\,q)_j$ for all $i,j\in\alpha$ and 
$\alpha=1,\ldots, p$.

In particular the set of interior equilibria of $X_{\nund,A}$ is the intersection of some affine subspace with ${\rm int}(\Gamma_\nund)$.
\end{proposition}


\section{Lotka-Volterra systems }
\label{lvs}

The standard sector 
\begin{displaymath}
\Rr_+^n=\{(x_1,\dots,x_n)\in \Rr^n : x_i\ge 0 , \quad \forall i\in\{1,\dots,n\} \}.
\end{displaymath}
is the phase space of Lotka-Volterra systems.

\begin{defn} We  call {\em Lotka-Volterra} (LV) to any system of differential equations on $\Rr^n_+$ of the form
\begin{equation}\label{eq1LV}
x_i' =x_i \left( r_i+\sum_{j=1}^n a_{ij}x_j \right), \quad i=1,\dots,n\;.
\end{equation}
\end{defn}

In the canonical interpretation~\eqref{eq1LV}  models the time evolution of an ecosystem with $n$  species. Each variable
$x_i$ represents the density of species $i$, the coefficient $r_i$ stands for the intrinsic rate of decay or growth of species $i$, and each coefficient $a_{ij}$ represents the effect of population $j$ over population $i$. For instance $a_{ij}>0$  means that population $j$ benefits population $i$. The matrix $A=(a_{ij})_{1\le i,j\le n}$ is called  the interaction matrix of system (\ref{eq1LV}).

The interior equilibria of ~\eqref{eq1LV} are the solutions $q\in \Rr^n_+$
of the non-homogeneous linear equation $r+A\,x=0$.
Given  $A\in\Mat_{n\times n}(\Rr)$ and   $q\in\Rr^n$ such that $r+A\,q=0$, the LV system~\eqref{eq1LV} can be written as
\begin{equation}\label{LV}
 \frac{dx}{dt}=X_{A,q}(x):=  x \ast A\,(x-q)\;,
\end{equation}
where   $\ast$ denotes the point-wise multiplication of vectors in $\Rr^n$.

\begin{defn}
\label{def lv dissip}
We say that the LV system~(\ref{LV}), the matrix $A$, or the vector field $X_{A,q}$,  is {\em dissipative} \, iff\, there is a positive diagonal matrix $D$ such that
$Q_{AD}(x)=x^T A D x\leq 0$ for every $x\in \Rr^n$.
\end{defn}

\begin{proposition}
\label{prop LV Lyapunov function}
If  $X_{A,q}$ is dissipative
 then, for any $D=\diag(d_i)$ as in Definition~\ref{def lv dissip},  $X_{A,q}$ admits the Lyapunov function
\begin{equation}\label{Liapunov_function}
h(x)=\sum_{i=1}^n \frac{ x_i-q_i\,\log x_i}{d_i} \;,
\end{equation}
which decreases along orbits of $X_{A,q}$.
\end{proposition}

\begin{proof}
The derivative of $h$ along orbits of $X_{A,q}$ is given by
\begin{align*}
  \dot{h}(x)   & =\sum_{i,j=1}^n \frac{a_{ij}}{d_i}(x_i-q_i)(x_j-q_j)
  =(x-q)^T D^{-1} A (x-q) \\
&=[D^{-1} (x-q)]^T A D [D^{-1} (x-q)] \le 0. 
\end{align*}
\end{proof}

We will denote by $\Nuc(A)$ the kernel of a matrix $A$.

\begin{proposition}
\label{prop nuc A =  nuc AT}
If $A\in \Mat_{n\times n}(\Rr)$ is dissipative and $D$ is a positive diagonal matrix such that $Q_{A D}\leq 0$ then \,  $\Nuc(A)=D\,\Nuc(A^T)$.
\end{proposition}

\begin{proof}
Assume first that $Q_A\leq 0$ on $\Rr^n$ and consider the decomposition $A = M + N$ with
$M=(A   + A^T)/2$ and $N=(A - A^T)/2$.
Clearly $\Nuc(M)\cap \Nuc(N)\subseteq \Nuc(A)$.
On the other hand, if $v\in\Nuc(A)$ then $v^T\, M\,v= v^T A\,v=0$.
Because $Q_M=Q_A\leq 0$ this implies that $M\,v=0$,
i.e., $v\in\Nuc(M)$. Finally, since $N=A-M$,  $v\in \Nuc(N)$.
This proves that  $\Nuc(A)=\Nuc(M)\cap \Nuc(N)$.
Similarly, one proves that $\Nuc(A^T)=\Nuc(M)\cap \Nuc(N)$.
Thus $\Nuc(A)=\Nuc(A^T)$. 

In general, if $Q_{A D}\leq 0$, we have
$\Nuc((A D)^T)= \Nuc( D A^T)= \Nuc(A^T)$,
and $\Nuc(A D)= D^{-1}\Nuc(A)$. Thus, from the previous case applied 
to $A D$ we get $D^{-1}\Nuc(A)=\Nuc(A^T)$.
\end{proof}

\begin{proposition}
\label{prop LV invariant foliation}
Any dissipative LV system admits an invariant foliation on ${\rm int}(\Rr^n_+)$ with a unique equilibrium point in each leaf.
\end{proposition}

\begin{proof}
See~\cite[Proposition 2.1 and Theorem 2.3]{DP2012}.
\end{proof}

On the rest of this section we focus attention on LV systems with interior equilbria $q\in {\rm int}(\Rr^n_+)$.
In this case the Lyapunov function $h$ is proper, and hence the forward orbits of ~\eqref{LV} are complete.
Therefore, the vector field $X_{A,q}$  induces a complete semi-flow $\phi_{A,q}^t$ on ${\rm int}(\Rr^n_+)$.

\begin{defn}\label{LV graph}
 Given a matrix $A=(a_{ij})\in\Mat_{n\times n}(\Rr)$ of a LV system,  we define its associated graph $G(A)$ to have  vertex set $\{1,\dots, n\}$, 
 and to contain an edge connecting vertex $i$ to vertex $j$\, iff \,$a_{ij}\neq0$ or $a_{ji}\neq0$.
\end{defn}

Given a matrix $A=(a_{ij})\in \Mat_{n\times n}(\Rr)$ we call \emph{admissible perturbation} of $A$ to any other matrix $\tilde{A}=(\tilde{a}_{ij})\in \Mat_{n\times n}(\Rr)$ such that
\begin{displaymath}
\tilde{a}_{ij}\approx a_{ij} \quad \text{ and } \quad
\tilde{a}_{ij}=0 \, \Leftrightarrow \, a_{ij}=0.
\end{displaymath}
By definition, admissible perturbation are perturbations of $A$ such that
$G(A)=G(\tilde{A})$.

\begin{defn}\label{stably dissip}
A matrix $A\in \Mat_{n\times n}(\Rr)$ is said to be \emph{stably dissipative}
if any small enough admissible perturbation  $\tilde{A}$ of $A$ is dissipative, i.e., if there exists $\varepsilon >0$ such that for any admissible perturbation
$\tilde{A}=(\tilde{a}_{ij})$ of $A=(a_{ij})$,
\begin{displaymath}
\max_{1\le i,j\le n}|a_{ij}-\tilde{a}_{ij}|<\varepsilon \, \Rightarrow \, \tilde{A} \textrm{ is dissipative}.
\end{displaymath}

A LV system~\eqref{LV}  is said to be \emph{stably dissipative} if its
 interaction matrix is  stably dissipative.
\end{defn}

\begin{lemma}\label{A sd and D positive implies AD sd}
Let $D$ be a positive diagonal matrix. If $A$ is a stably dissipative matrix, then
$AD$ and $D^{-1} A$ are also stably dissipative.
\end{lemma}

\begin{proof}
Since $A$ is dissipative there exists a positive diagonal matrix $D'$ such that
$ Q_{AD'}\leq 0\,, $ which is equivalent to $ Q_{(AD)(D^{-1}D')}\le 0\,.$
Hence  $AD$ is dissipative.
Analogously, since $Q_{A D'}\leq 0$ we have
$Q_{D^{-1}A D' D^{-1}} (x)=Q_{A D'}(D^{-1}x)\leq 0$,
which shows that $D^{-1} A$ is dissipative.

Let $B$ be a small enough admissible perturbation of $AD$. Then there exists an
admissible perturbation  $\tilde{A}$ of $A$ such that
$B=\tilde{A} D$. Since $A$ is stably dissipative the matrix $\tilde{A}$ is
dissipative as well. Hence  there exists a positive diagonal matrix $D''$ such that
$ Q_{\tilde{A} D''}\leq 0\,, $ which is equivalent to
$ Q_{(\tilde{A} D)(D^{-1}D'')}\le 0 $.
This proves that  $B=\tilde{A}D$ is dissipative. Therefore $AD$ is stably dissipative.

A similar argument proves that
$D^{-1}A$ is stably dissipative.
\end{proof}

\begin{defn}
Given a matrix $A\in \Mat_{n\times n}(\Rr)$ and a subset $I\subseteq\{1,\dots,n\}$, we say that
$A_I=(a_{ij})_{(i,j)\in I\times I}$ is the {\em submatrix} $I\times I$ of $A$.
\end{defn}

\begin{lemma}\label{lemma 3.1.5. masterthesis}
Let $A\in \Mat_{n\times n}(\Rr)$ be a stably dissipative matrix. Then, for all $I\subseteq\{1,\dots,n\}$, the submatrix $A_I$ is stably dissipative.
\end{lemma}

\begin{proof}
Let $I\subset \{1,\dots,n\}$ and consider
an admissible perturbation  $B=(b_{ij})_{i,j\in I}$  of $A_I$.
Define $\tilde{A}=(\tilde a_{ij})$  to be the matrix with entries 
$$ \tilde a_{ij}=\left\{ 
\begin{array}{lll}
b_{ij} & \text{ if } & (i,j)\in I\times I\\
a_{ij} & \text{ if } & (i,j)\notin I\times I
\end{array}
\right. . $$
Clearly, $\tilde{A}$ is an admissible perturbation of $A$. Hence there exists a positive diagonal matrix $D$ such that $\tilde{A}D\le 0$. Letting now
$D_I$ be the $I\times I$  submatrix of $D$,   we see that $B D_I =(\tilde A\, D)_I \le 0$, which concludes the proof.
\end{proof}

\begin{defn}
We call attractor of the LV system~\eqref{LV}
to the following topological closure
 $$\attr{A,q} :=\overline{ \cup_{x\in \Rr^ n_+} \omega(x) }\;, $$
 where $\omega(x)$ is the $\omega$-limit of $x$ by the semi-flow
 $\{\phi_{A,q}^t:\Rr^n_+\to\Rr^n_+\}_{t\geq 0}$. 
\end{defn}

We need the following classical theorem
(see~\cite[Theorem 2]{LaS1968}).
\begin{theorem}[La Salle]
\label{La Salle theorem}
Given a vector field $f(x)$ on  a manifold $M$,
consider the autonomous  o.d.e. on $M$,
\begin{equation}
\label{ode}
x' = f(x) .
\end{equation}  
Let $h:M\to\Rr$ be a smooth function such that
\begin{enumerate}
\item $h$ is a Lyapunov function, i.e., the derivative of $h$ along the flow satisfies 
$\dot h(x):= D h_x f(x)\leq 0$ for all $x\in M$.
\item $h$ is bounded from below.
\item $h$ is a proper function, i.e.
$ \{ h\leq a\}$ is compact for all $a\in\Rr$.
\end{enumerate}
Then~\eqref{ode} induces a complete semi-flow on $M$ such that the topological closure of all its $\omega$-limits is contained in the region where the derivative of $h$ along the flow  vanishes, i.e.,
$$   \overline{\cup_{x\in M} \omega(x)}
\subseteq \{ x\in M \colon \dot h(x)=0 \}.$$
\end{theorem}


The following lemma plays a key role in the theory of stably dissipative systems.

\begin{lemma}\label{lema2.1.ZL2010}
Given a stably dissipative matrix $A$, if $D$ is a positive
diagonal matrix $D$ such that $Q_{AD}\leq 0$ then for all $i=1,\dots, n$
and $w\in\Rr^n$,
$$ Q_{A D}(w) = 0 \quad \Rightarrow\quad a_{ii}\,w_i=0 \; .$$
\end{lemma}

 \begin{proof}
See~\cite{RZ1982}.
 \end{proof}

By Theorem~\ref{La Salle theorem}
the attractor $\attr{A,q}$ is contained in the
set $\{ \dot h=0\}$.
By the proof Proposition~\ref{prop LV Lyapunov function} we have \,  $\dot h(x) =  Q_{D^{-1}A}(x-q)$. Hence
$$\attr{A,q}\subseteq \left\{  x\in\Rr^n_+ \colon  Q_{D^{-1}A}(x-q)=0 \right\} ,$$
and  by Lemma~\ref{lema2.1.ZL2010} it follows that $\attr{A,q}\subseteq \{ x \colon x_i=q_i\}$ for every $i=1,\ldots, n$ such that
$a_{ii}<0$. 

\bigskip

Let us say that a species $i$ {\em is of type $\blV$}  to mean that the following inclusion holds  $\attr{A,q}\subseteq \{ x\,:\, x_i=q_i\}$.
Similarly, we say that a species $i$ {\em is of type $\oplus$}, to state that 
$\attr{A,q}\subseteq \{ x\,:\, X_{A,q}^i(x)=0 \}$,
where $X_{A,q}^i(x)$ stands for the $i$-th component of the vector $X_{A,q}(x)$. Equivalently,
the strategy $i$ is of type $\oplus$ if and only if the sets  $\{x_i={\rm const}\}$ are  invariant under the flow $\phi_{A,q}^t:\attr{A,q}\hookleftarrow$.
With this terminology it can be proven that

\begin{proposition} Given neighbor vertexes $j,l$ in the graph $G(A)$,
\begin{itemize}
	\item[(a)] If  $j$ is of type $\bullet$ or $\oplus$  and all of its neighbors are of type $\bullet$, except for  $l$, then  $l$ is of type $\bullet$;
	\item[(b)] If  $j$ is of type $\bullet$ or $\oplus$  and all of its neighbors are of type $\bullet$ or $\oplus$, except for  $l$, then   $l$ is of type $\oplus$;
	\item[(c)] If  all  neighbors of $j$ are of type $\bullet$ or $\oplus$, then $j$ is of type $\oplus$.
\end{itemize}
\end{proposition}

\begin{proof}
See~\cite{RW1984}.
\end{proof}

Based on these facts, Redheffer \textit{et al.} introduced a reduction algorithm on the graph $G(A)$ to derive information on the species' types of a stably dissipative 
 LV system~\eqref{LV}.

\begin{Rule}
Initially, colour black, $\bullet$, every vertex $i$ such that  $a_{ii}<0$, and colour white, $\circ$, all other vertices.
\end{Rule}

The reduction procedure consists in applying the following  rules,
 corresponding to valid inference rules:

\begin{Rule}
 If  $j$ is a $\bullet$ or $\oplus$-vertex and all of its neighbours are $\bullet$, except for one vertex $l$, then   colour $l$ as $\bullet$;
\end{Rule}

\begin{Rule}
If  $j$ is a $\bullet$ or $\oplus$-vertex and all of its neighbours are $\bullet$ or $\oplus$, except for one vertex $l$, then draw $\oplus$ at the vertex $l$;
\end{Rule}

\begin{Rule}
If  $j$ is a $\circ$-vertex and all of is neighbours are $\bullet$ or $\oplus$, then   draw $\oplus$ at the vertex $j$.
\end{Rule}

 Redheffer \textit{et al.}  define the  {\em reduced graph} of the system,   $\Rscr(A)$, 
as the graph  obtained from $G(A)$ by successive applications of the reduction rules  2-4,  until they can no longer be applied.
An easy consequence of this theory is the following result.

\begin{proposition}
Let $A\in\Mat_n(\Rr)$ be a stably dissipative matrix
and consider the LV system~\eqref{LV} with an equilibrium
$q\in {\rm int}(\Rr^n_+)$.

\begin{enumerate}
\item If all vertices of $\Rscr(A)$ are $\bullet$ then 
$q$ is the unique globally attractive equilibrium.
\item If $\Rscr(A)$ has only $\bullet$ or $\oplus$ vertices then 
there exists an invariant foliation with a unique globally attractive equilibrium in each leaf.
\end{enumerate}
\end{proposition}

\begin{proof}
Item (1) is clear because if all vertices are of type $\bullet$
then for every orbit $x(t)=(x_1(t),\ldots, x_n(t))$ of ~\eqref{LV}, and every $i=1,\ldots, n$, one has\, 
$\lim_{t\to +\infty} x_i(t)=q_i$.

Likewise, if $\Rscr(A)$ has only $\bullet$ or $\oplus$ vertices
then every orbit of ~\eqref{LV} converges to an equilibrium point, which depends on the initial condition.
But by Proposition~\ref{prop LV invariant foliation} there exists an
invariant foliation $\Fscr$ with a single equilibrium point in each leaf. Hence, the unique equilibrium point in each leaf of $\Fscr$ must be globally attractive. 
\end{proof}

\begin{defn}
We say that a dissipative matrix $ A\in \Mat_{n\times n}(\Rr)$ is {\em almost skew-symmetric} \, iff\, $a_{ij}=-a_{ji}$ whenever $a_{ii}=0$ or $a_{jj}=0$,
and the quadratic form $Q_A$ is negative definite on the subspace 
$$E=\{\, w\in\Rr^n\,\colon\, 
w_i=0 \; \text{ for all }\; i\; \text{ such that  }\; a_{ii}=0\,\}\;.$$
\end{defn}

\begin{defn}\label{def strong link}
We say that the  graph $G(A)$ has a {\em strong link} $(\bbEdge)$ if there is
an edge $\{i,j\}$ between vertexes $i,j$ such that $a_{ii}<0$ and $a_{jj}<0$.
\end{defn}

\begin{proposition}[Zhao-Luo~\cite{ZL2010}]\label{sd:char}
Given   $ A\in \Mat_{n\times n}(\Rr)$, 
$A$ is stably dissipative \, iff\,  every cycle of $G(A)$ contains at least a  strong link and there is a positive diagonal matrix $D$ such that $A D$ is almost skew-symmetric.
\end{proposition}

\begin{proof}
See~\cite[Theorem 2.3]{ZL2010}, or~\cite[Proposition 3.5]{DP2012}.
\end{proof}

 A  compactification procedure introduced by J. Hofbauer~\cite{Hof1981} 
 shows that every Lotka-Volterra system in $\Rr_+^n$ is orbit equivalent to a replicator system
 on the $n$-dimensional simplex $\Delta^n$. We briefly recall this compactification.
 Let $A$ be a  $n\times n$ real matrix and $r\in\Rr^n$ a constant vector. The Lotka-Volterra equation associated to $A$ and $r$ is defined on $\Rr^n_+$ as follows
\begin{equation}\label{lv} \frac{d z_i}{dt} = z_i\left( \, r_i+(Az)_i\right) 
\quad  1\leq i\leq n \;.
\end{equation}
For each $j=1,\ldots, n+1$, let $\sigma_{j}:=\{\,x\in\Delta^n\subset\Rr^{n+1}\,\colon\, x_{j}=0\,\}$ and consider the diffeomorphism 
 \begin{align*}
 \phi:\Rr^n_+&\to \Delta^n\backslash\sigma_{n+1}\\
 (z_1,\ldots,z_n)&\mapsto\frac{1}{1+\sum\limits_{i=1}^nz_i}(z_1,\ldots,z_n,1).
 \end{align*}
A straightforward calculation  shows that the push-forward of the vector field~\eqref{lv} is equal to   $\frac{1}{x_{n+1}}X_{\tilde{A}}$. where $X_{\tilde{A}}$ is the replicator vector field associated to the payoff matrix
$$ \tilde{A}=\begin{pmatrix}a_{11}&\ldots&a_{1n}&r_1\\\vdots&\vdots&\vdots&\vdots\\a_{n1}&\ldots&a_{nn}&r_n\\0&\ldots&0&0\end{pmatrix}\;.$$
Since the flows of $\frac{1}{x_{n+1}}X_{\tilde{A}}$  and  $X_{\tilde{A}}$ are orbit equivalent, we refer to $X_{\tilde{A}}$ as the compactification of the  LV equation~\eqref{lv}.


\section{Hamiltonian Polymatrix Replicators}
\label{hamiltonian}

\begin{defn}\label{defi:formal_equilibrium}
We say that any vector $q\in \Rr^n$ is a \emph{formal equilibrium} of a polymatrix game $(n,A)$ if
\begin{itemize}
	\item[(a)] $(Aq)_i=(Aq)_j$ for all $i,j\in\alpha$, and all $\alpha=1,\dots , p$,
	\item[(b)] $\sum_{j\in\alpha}q_j =1$ for all $\alpha=1,\dots , p$.
\end{itemize}
\end{defn}

The matrix $A$ induces a quadratic form
$Q_A:H_{\nund}\to\Rr$ defined by $Q_A(w):=w^T\,A\, w$,
where $H_\nund$ is defined in~\eqref{def Hnund}.

\begin{defn}
We call diagonal matrix of type $\nund$
to any  diagonal matrix  $D=\diag(d_{i})$ such that  $d_{i}=d_{j}$ for all  $i,j\in\alpha$  and  $\alpha=1,\ldots, p$.
\end{defn}

\begin{defn} \label{def cons 1}
A polymatrix game $(\nund,A)$ is called \emph{conservative}  if it
has a formal equilibrium $q$, and there exists a positive diagonal
matrix $D$ of type $\nund$ such that $Q_{A\,D} = 0$ on $H_{\nund}$.
\end{defn}

In~\cite{AD2015} we have defined conservative 
polymatrix game as follows.

\begin{defn} \label{def cons 2}
  A polymatrix game  $(\nund,A)$ is called {\em conservative}  if
  \begin{enumerate}
\item[(a)] it has a formal equilibrium, 
\item[(b)] there are matrices $A_0, D\in\Mat_{n\times n}(\Rr)$ such that 
\begin{enumerate}
\item[(i)] $(\nund, A)\sim (\nund, A_0 D)$,
\item[(ii)] $A_0$ is skew-symmetric, 
\item[(iii)] $D$ is a positive diagonal matrix of type $\nund$.
\end{enumerate}
\end{enumerate}
\end{defn}

However, we will prove in Proposition~\ref{prop: def equiv} that  these two definitions are equivalent.

Let $\{e_1,\ldots, e_n\}$ denote  the canonical basis in $\Rr^n$, and  $V_{\nund}$ be the set of vertices of $\Gamma_{\nund}$.
Each vertex $v\in V_{\nund}$ can be written as
$v=e_{i_1}+\cdots + e_{i_p}$, with $i_{\alpha}\in\alpha$, $\alpha =1,\ldots, p$, and it determines the set
$$ \Vscr_v :=\{\, (i, i_{\alpha})\,\colon\, i\in\alpha,\; i\neq i_{\alpha},\; \alpha =1,\ldots, p\, \} $$
of cardinal $n-p=\dim(H_{\nund})$. 
Notice that   $(i,j)\in\Vscr_v$ \, iff \,
$i\neq j$ are in the same group and $v_j=1$.
Hence there is a natural identification
$\Vscr_v\equiv \{\, i\in \{1,\ldots, n\}\,\colon\, v_i=0\,\}$.
For every vertex $v$, the family $\Bscr_v:=\{\, e_i-e_j\,\colon\, (i,j)\in \Vscr_v \,\}$ is a basis of $H_{\nund}$.

\begin{lemma}\label{lemma x-q}
For any vertex $v$ of $\Gamma_{\nund}$ and $x,q\in \Gamma_{\nund}$,
$$ x-q = \sum_{(i,j)\in \Vscr_v} (x_i-q_i)\,(e_i-e_j) \;. $$
\end{lemma}

\begin{proof}
Let $v$ be a vertex of $\Gamma_{\nund}$. Notice that for all $\alpha=1,\ldots, p$,
$$ 
-(x_{i_\alpha}-q_{i_\alpha}) = 
\sum_{\substack{ i\neq i_\alpha\\ i\in\alpha }}\left(x_i-q_i\right) \,.$$

\begin{align*}
&\sum_{(i,j)\in \Vscr_v} (x_i-q_i)\,(e_i-e_j)   =
\sum_{\alpha=1}^p \sum_{\substack{ i\neq i_\alpha \\ i\in\alpha }} (x_i-q_i) (e_i-e_{i_\alpha}) \\
&\quad = \sum_{\alpha=1}^p \sum_{\substack{ i\neq i_\alpha \\ i\in\alpha }} (x_i-q_i)  e_i  -
\sum_{\alpha=1}^p \sum_{\substack{ i\neq i_\alpha \\ i\in\alpha }} (x_i-q_i)  e_{i_\alpha}  \\
&\quad = \sum_{\alpha=1}^p \sum_{\substack{ i\neq i_\alpha \\ i\in\alpha }} (x_i-q_i)  e_i  +
\sum_{\alpha=1}^p (x_{i_\alpha}- q_{i_\alpha}) e_{i_\alpha}  \\
&\quad = \sum_{\alpha=1}^p \sum_{  i\in\alpha } (x_i-q_i)  e_i = x- q
\end{align*} 
\end{proof}

Given ordered pairs of strategies in the same group $(i,j),(k,\ell)$, i.e.,
$i,j\in\alpha$ and $k,\ell\in\beta$ for some $\alpha,\beta\in \{1,\ldots, p\}$, define
$$A_{(i,j),(k,\ell)}:=
  a_{ik} + a_{j\ell}
-  a_{i\ell} - a_{jk} \;. $$

\begin{proposition}
The  coefficients $A_{(i,j),(k,\ell)}$ do not depend on the representative $A$ of
the polymatrix game $(\nund,A)$.
\end{proposition}

\begin{proof}
Consider the matrix $B=A-C$, where the blocks
 $C^{\alpha\beta}=\left(c_{ij}\right)_{i\in\alpha,j\in\beta}$ of $C$ have equal
rows for all $\alpha,\beta=1,\dots,p$.
Let $(i,j)\in\alpha$ and $(k,\ell)\in\beta$ with $\alpha,\beta\in\{1,\dots,p\}$. Then
\begin{eqnarray}
B_{(i,j),(k,\ell)} \nonumber &=& b_{ik}+b_{j\ell}-b_{i\ell}-b_{jk} \\ \nonumber
		&=&a_{ik}-c_k+a_{j\ell}-c_{\ell}-a_{i\ell}+c_{\ell}-a_{jk}+c_k \\ \nonumber
		&=& A_{(i,j),(k,\ell)} \,, \nonumber
\end{eqnarray}
where $c_k$ is the constant entry on the $k^{th}$-column of $C^{\alpha\beta}$.
\end{proof}

\begin{defn}\label{def Av}
Given $v\in V_{\nund}$, we define $A_v\in \Mat_{d\times d}(\Rr)$, $d=n-p$, to be the matrix  with  entries $A_{(i,j),(k,\ell)}$,
indexed in $\Vscr_v\times \Vscr_v$, and
$G(A_v)$ to be its associated graph (see Definition~\ref{LV graph}).
\end{defn}

\begin{proposition} \label{DQA rule}
The matrix $A_v$ represents the quadratic form \linebreak
$Q_A:H_{\nund}\to\Rr$ in the basis $\Bscr_v$.

More precisely, if $q$ is a formal equilibrium of  the polymatrix game $(\nund,A)$ then the quadratic form $Q_A:H_\nund \to\Rr$ is given by
\begin{equation}\label{quadratic form x-q}
Q_A(x-q) = \sum_{(i,j),(k,\ell)\in \Vscr_v} A_{(i,j),(k,\ell)}\,(x_i-q_i)\,(x_k-q_k) \;.
\end{equation}
\end{proposition}

\begin{proof}
Using lemma~\ref{lemma x-q}, we have
\begin{eqnarray}
Q_A(x-q) \nonumber &=& \left( \sum_{(i,j)\in\Vscr_v}(x_i-q_i)(e_i-e_j)\right)^T A\left( \sum_{(k,\ell)\in\Vscr_v}(x_k-q_k)(e_k-e_\ell) \right) \\ \nonumber
	&=&\sum_{(i,j),(k,\ell)\in\Vscr_v}(e_i-e_j)^TA(e_k-e_\ell)(x_i-q_i)(x_k-q_k) \\ \nonumber
	&=&\sum_{(i,j),(k,\ell)\in\Vscr_v}A_{(i,j),(k,\ell)}(x_i-q_i)(x_k-q_k) \,, \nonumber
\end{eqnarray}
\end{proof}

\begin{remark}
All matrices $A_v$, with $v\in V_{\nund}$,
have the same rank because they represent, in different bases,  the same (non-symmetric) bilinear form
$B_A:H_\nund \times H_\nund  \to\Rr$, $B_A(v,w):= v^T\,A\,w$. 
\end{remark}

\begin{proposition}\label{prop: def equiv}
Definitions~\ref{def cons 1}  and~\ref{def cons 2}  are equivalent.
\end{proposition}

\begin{proof}

Given a matrix $C$  with blocks $C^{\alpha\beta}=\left(c_{ij}\right)_{i\in\alpha,j\in\beta}$ having equal rows
for all $\alpha,\beta=1,\dots,p$, it is clear that
$C_{(i,j),(k,\ell)}=0$ for all pairs of strategies
$(i,j)$, $(k,\ell)$ in the same group.
Hence, by Proposition~\ref{DQA rule},
$Q_C$ vanishes on $H_\nund$.

If $(\nund,A)$ is conservative in the sense of
Definition~\ref{def cons 2} then there are matrices:
$A_0$ skew-symmetric, and $D$ positive diagonal of type $\nund$, such that $(\nund, A)\sim (\nund, A_0 D)$.
It follows that
$(\nund, A D^{-1})\sim (\nund, A_0)$ and as observed  above the matrix $C= A D^{-1}-A_0$ satisfies $Q_{C}=0$ on $H_\nund$.
Finally, since $A_0$ is skew-symmetric, we have
$Q_{ A D^{-1}} =0$ on $H_\nund$.
In other words, $(\nund,A)$ is conservative in the sense of Definition~\ref{def cons 1}.

Conversely, assume that $A$ is  conservative in the sense of Definition~\ref{def cons 1}.
Then for some positive diagonal matrix $D$ of type $\nund$, $Q_{A D^{-1}}$ vanishes on $H_\nund$.

Let $\{v_1,\ldots, v_n\}$ be an orthonormal basis of $\Rr^n$ where the vectors $v_\alpha=\frac{1}{\sqrt{n_\alpha}}\,\pi_\alpha(\one)$, with $\alpha\in\{1,\ldots, p\}$,
form a orthonormal basis of $H_\nund^\perp$,
and the family $\{v_{p+1},\ldots, v_n\}$ is any orthonormal basis of $H_\nund$.

Let $m_{ij}=\langle A D^{-1} v_i , v_j\rangle$, for
all $i,j=1,\ldots, n$, so that $M=(m_{ij})_{i,j}$
represents the linear endomorphism
$A D^{-1}:\Rr^n\to \Rr^n$ w.r.t. the basis $\{v_1,\ldots, v_n\}$. Since $Q_{A D^{-1}}=0$ on $H_\nund$, the $(n-p)\times (n-p)$ sub-matrix $M'$ of $M$, formed by the
last $n-p$ rows and columns of $M$, is skew-symmetric.

Let $M_0\in\Mat_{n\times n}(\Rr)$ be a skew-symmetric matrix that shares  with $M$ its last $n-p$ rows. Let $A_0:\Rr^n \to \Rr^n $ be the linear endomorphism represented by the matrix $M_0$ w.r.t. the basis $\{v_1,\ldots, v_n\}$, and identify $A_0$ with the matrix that represents it w.r.t. the canonical basis.
Because $M_0$ is skew-symmetric, and $\{v_1,\ldots, v_n\}$ orthonormal, $A_0$ is skew-symmetric too.

Then $C=A\,D^{-1}-A_0$ is represented by the matrix
$M-M_0$ w.r.t. the basis $\{v_1,\ldots, v_n\}$.
Since the last $n-p$ rows of $M-M_0$ are zero,
the range of $C:\Rr^n\to\Rr^n$ is contained in $H_\nund^\perp$. Hence, by Lemma~\ref{lemma equivalence},
$(\nund, A D^{-1})\sim (\nund, A_0)$, which implies
$(\nund, A)\sim (\nund, A_0 D)$. Since $A_0$ is skew-symmetric, this proves that $(\nund,A)$ is conservative in the sense of Definition~\ref{def cons 2}.
\end{proof}

\begin{remark}
\label{rmk AD D-1A}
For all $w\in H_\nund$, 
$Q_{D^{-1}A}(w)=Q_{A D}(D^{-1} w)$.
Hence, because $D\, H_{\nund}= H_{\nund}$ for any diagonal matrix $D$ of type $\nund$
\begin{enumerate}
\item $Q_{A\,D}(w) = 0$ \; $\forall\; w\in H_\nund$\quad 
$\Leftrightarrow$ \quad 
$Q_{D^{-1}A}(w) =  0$ \; $\forall\; w\in H_\nund$. 
\item $Q_{A\,D}(w) \leq 0$ \; $\forall\; w\in H_\nund$\quad 
$\Leftrightarrow$ \quad 
$Q_{D^{-1}A}(w) \leq  0$ \; $\forall\; w\in H_\nund$. 
\end{enumerate}
\end{remark}

\begin{lemma}\label{lemma dot h}
Given $A\in\Mat_{n\times n}(\Rr)$, if  $q$ is a formal equilibrium of $X_{\nund,A}$, and $D=\diag(d_i)$ is a positive diagonal matrix  of type $\nund$,
then the derivative of 
\begin{equation}\label{hamilt}
h(x)=-\sum_{i=1}^n \frac{q_i}{d_i}\,\log x_i
\end{equation} 
along the flow of $X_{\nund,A}$ satisfies
$$\dot h (x) =  Q_{D^{-1}A}(x-q)\;. $$
\end{lemma}

\begin{proof}
\begin{eqnarray}\nonumber
\dot h  &=&-\sum_{\alpha=1}^p\sum_{i\in \alpha} \frac{q_i}{d_i}\frac{\dot{x_i}}{x_i}=-\sum_{\alpha=1}^p\sum_{i\in \alpha} \frac{q_i}{d_i} \left((Ax)_i-\sum_{\beta=1}^p (x^{\alpha})^ t A^{\alpha,\beta} x^{\beta}\right)\\ \nonumber
		&=&-q^TD^{-1}Ax+x^TD^{-1}Ax=(x-q)^TD^{-1}Ax\\ \nonumber
		&=&(x-q)^TD^{-1}Ax-\underbrace{(x-q)^TD^{-1}Aq}_{=0}\\ \nonumber
		&=&(x-q)^TD^{-1}A(x-q)=Q_{D^{-1}A}(x-q) \,.
\end{eqnarray}
To explain the vanishing term notice that
for all $\alpha\in \{1,\ldots, p\}$ and $i,j\in\alpha$,
$(A\,q)_i=(A\,q)_j$, $d_i=d_j$ and $\sum_{k\in\alpha} (x_k-q_k)=0$.
\end{proof}

\begin{proposition}\label{Conserv_Poly}
If $A$ is conservative, and $q$ and $D=\diag(d_i)$ are as in Definition~\ref{def cons 1}, then~\eqref{hamilt} 
is a first integral for the flow of $X_{\nund,A}$, i.e.,
$\dot{h} = 0$ along the flow of $X_{\nund,A}$. 

Moreover,
$X_{\nund,A}$ is Hamiltonian w.r.t. a stratified Poisson structure on the prism $\Gamma_\nund$, having $h$ as its Hamiltonian function.
\end{proposition}

\begin{proof}
The first part follows from Lemma~\ref{lemma dot h} and Remark~\ref{rmk AD D-1A}.
The second follows from ~\cite[theorem 3.20]{AD2015}.
\end{proof}


\section{Dissipative Polymatrix Replicators}
\label{dissipative}

\begin{defn}\label{defi:dissipative_poly}
A polymatrix game $(\nund,A)$ is called  {\em dissipative}\index{polymatrix game!dissipative}  if it
has a formal equilibrium $q$, and there exists a positive diagonal
matrix $D$ of type $\nund$ such that $Q_{A\,D}\leq 0$ on $H_{\nund}$.
\end{defn}

\begin{proposition}\label{prop Lyapunov function}
If $(\nund,A)$ is dissipative, and $q$ and $D$ are as in Definition~\ref{defi:dissipative_poly}, then
$$ h(x)=-\sum_{i=1}^n \frac{q_i}{d_i}\,\log x_i $$
is a Lyapunov decreasing function for the flow of $X_{\nund,A}$, i.e.,
$\frac{dh}{dt}\leq 0$ along the flow of $X_{\nund,A}$.
\end{proposition}

\begin{proof}
Follows from Lemma~\ref{lemma dot h}, and Remark~\ref{rmk AD D-1A}.
\end{proof}

\newcommand{\Repr}[2]{{#2}^{(#1)}}
\newcommand{\SymRepr}[2]{S {#2}^{(#1)}}
\newcommand{\ReprEntry}[3]{{#2}^{(#1)}_{#3}}
\newcommand{\SymReprEntry}[3]{S {#2}^{(#1)}_{#3}}

\begin{defn}\label{defi:admissible_poly}
A polymatrix game $(\nund,A)$ is called {\em admissible}\index{polymatrix game!admissible} if $A$ is dissipative and for some vertex $v\in\Gamma_{\nund}$ the matrix $A_v$ is stably dissipative (see Definition~\ref{stably dissip}). We denote by $V_{\nund,A}^*$ the subset of vertices $v\in V_{\nund}$ such that $A_v$ is stably dissipative.
\end{defn}

\begin{proposition}\label{prop:quotient rule}
Let $q$ be a formal equilibrium of  the polymatrix game $(\nund,A)$.
Given $v\in V_{\nund}$ and $(i,j)\in \Vscr_v$, then we have the following  quotient rule 
\index{quotient rule!polymatrix}
\begin{equation}\label{poly quotient rule}
\frac{d}{dt}\left(\frac{x_i}{x_j}\right)=
\frac{x_i}{x_j}\, \sum_{(k,\ell)\in \Vscr_v}A_{(i,j),(k,\ell)}\, (x_k-q_k) \;.
\end{equation}
\end{proposition}

\begin{proof}
Let $v$ be a vertex of $\Gamma_{\nund}$, $(i,j)\in \Vscr_v$, and $q$ be a formal equilibrium.
Using Lemma~\ref{lemma x-q}, we have
\begin{eqnarray}
\frac{d}{dt}\left(\frac{x_i}{x_j}\right) \nonumber &=& \frac{x_i}{x_j}\left((Ax)_i-(Ax)_j\right) \\ \nonumber
   &=& \frac{x_i}{x_j}\left(\left(A(x-q)\right)_i-\left(A(x-q)\right)_j\right) \\ \nonumber
   &=& \frac{x_i}{x_j}\sum_{(k,\ell)\in\Vscr_v}(e_i-e_j)^TA(e_k-e_\ell)(x_k-q_k) \\ \nonumber
   &=& \frac{x_i}{x_j}\sum_{(k,\ell)\in\Vscr_v} A_{(i,j),(k,\ell)} (x_k-q_k) \,. \nonumber
\end{eqnarray}
\end{proof}

\begin{proposition}\label{prop:limit quot rule}
If the dissipative polymatrix replicator associated to $(\nund,A)$
has an equilibrium $q\in\inter\left(\Gamma_{\nund}\right)$,
then for any state $x_0\in\inter\left(\Gamma_{\nund}\right)$  and any pair of strategies $i,j$ in the same group, the solution $x(t)$ of~\eqref{ode:pmg} with initial condition $x(0)=x_0$
satisfies 
$$ \frac{1}{c}\le \frac{x_i(t)}{x_j(t)}\le c\,, \quad \textrm{for all }\quad t\ge 0\,,$$ 
where  $c=c(x)$ is a constant depending on $x$.
\end{proposition}

\begin{proof}
Notice that the Lyapunov function $h$ in Proposition~\ref{prop Lyapunov function} is a proper function because $q\in {\rm int}(\Gamma_{\nund})$.
Given $x_0\in\inter\left(\Gamma_{\nund}\right)$,   $h(x_0)=a$ for some constant  $a>0$. By Proposition~\ref{prop Lyapunov function} the compact set $K=\{ x\in  {\rm int}(\Gamma_{\nund}) \colon h(x)\le a\}$ is forward invariant by the flow of $X_{\nund,A}$. In particular, the solution of the polymatrix replicator with initial condition $x(0)=x_0$ lies in $K$. Hence the quotient $\frac{x_i}{x_j}$ has a minimum and a maximum in $K$.
\end{proof}

\begin{proposition}
\label{prop invariant foliation}
Given a dissipative polymatrix game  $(\nund,A)$, if $X_{\nund,A}$ admits an equilibrium $q\in {\rm int}(\Gamma_{\nund})$ then there exists a
$X_{\nund,A}$-invariant foliation $\Fscr$ on ${\rm int}(\Gamma_\nund)$ such that every leaf of $\Fscr$ contains exactly one equilibrium point.
\end{proposition}

\begin{proof}
Fix some  vertex $v\in V_\nund$.
Recall that the entries of $A_v$ are indexed in the set $\Vscr_v\equiv \{\, i\in \{1,\ldots, n\}\,\colon\, v_i=0\,\}$.
Given a vector $w=(w_i)_{i\in \Vscr_v}\in\Rr^{n-p}$, we denote by $\bar w$ the unique vector   $\bar w \in H_\nund$ such that $\bar w_i=w_i$ for all
$i\in \Vscr_v$.

Let $\Escr\subset \Rr^n$ be the affine subspace of all
points $x\in\Rr^n$ such that for all $\alpha=1,\ldots, p$ and all $i,j\in\alpha$,
$(A\,x)_i=(A\,x)_j$ and $\sum_{j\in\alpha} x_j=1$. By definition 
$\Escr\cap {\rm int}(\Rr^n)$ is the set of interior equilibria of $X_{\nund,A}$.
We claim that $\Escr=\{ q+\bar w\colon w\in \Nuc(A_v)\}$. To see this it is enough to remark that
$w\in \Nuc(A_v)$ if and only if 
$$(A\bar w)_i-(A \bar w)_j= (e_i-e_j)^T A\bar w=0\quad
\forall\, (i,j)\in\Vscr_v.$$

Given $b\in \Nuc(A_v^T)$,
consider the function $g_b:{\rm int}(\Rr^n_+)\to \Rr$ defined by
 $g_b(x):= \sum_{j=1}^n \bar b_j\,\log x_j$.
 The restriction of $g_b$ to $\Gamma_\nund$  is invariant by the flow of $X_{\nund,A}$.
Note we can write
$$ g_b(x)= \sum_{l=1}^n \bar b_l \,\log x_l =
\sum_{(i,j)\in \Vscr_v} b_i\,\log \left(\frac{x_i}{x_j}\right) ,$$ and differentiating $g_b$ along the flow of $X_{\nund,A}$, by Proposition~\ref{prop:quotient rule}   we get 
$$ \dot g_b(x) =b^T\, A_v (x_k-q_k)_{k\in\Vscr_v}=0 \quad \text{ for all }\; x\in\Gamma_\nund.$$

Fix a basis $\{b_1,\ldots, b_k\}$ of $\Nuc(A_v^T)$,
and define  $g:{\rm int}(\Rr^n_+)\to \Rr^k$ by
$g(x):=(g_{b_1}(x),\ldots, g_{b_k}(x))$.
This map is a submersion.
For that consider the matrix $B\in\Mat_{k\times n}(\Rr)$
whose rows are the vectors $\bar b_j$, $j=1,\ldots, k$.
We can write
$g(x)=B\,\log x$, where $\log x= (\log x_1,\ldots, \log x_n)$.
Hence $D g_x = B\,D_x^{-1}$, where $D_x=\diag(x_1,\ldots, x_n)$,
and because $B$  has maximal rank, $\rank(B)=k$, 
 the map $g$ is a submersion. Hence $g$ determines the foliation $\Fscr$ whose leaves are the pre-images $g^{-1}(c)=\{ g\equiv c\}$ with $c\in\Rr^k$.

Let us now explain why each leaf of $\Fscr$
contains exactly one  point in $\Escr$.  
Consider  the vector subspace
 parallel to $\Escr$,  $E_0:=\{\bar w \colon w\in \Nuc(A_v)\}$. 
Because $(\nund, A)$ is dissipative, $A_v\in\Mat_{d\times d}(\Rr)$, $d=n-p$, is also dissipative, and by Proposition~\ref{prop nuc A = nuc AT}, $\Nuc(A_v)$ and $\Nuc(A_v^T)$ have the same rank.
Therefore $\dim (E_0)=k$. Let $\{c_1,\ldots, c_{n-k}\}$ be a basis of
$E_0^\perp\subset \Rr^n$ and consider the matrix 
$C\in\Mat_{(n-k)\times n}(\Rr)$ whose rows are the vectors $c_j$,
$j=1,\ldots, n-k$. The matrix $C$  provides the following description  
$ \Escr=\{ x\in\Rr^n \colon C\,(x-q)=0 \}$.
Consider the matrix  $U=\left[\begin{array}{c}
B \\
\hline
C
\end{array}\right]\in\Mat_{n\times n}(\Rr)$, 
which is nonsingular because by Proposition~\ref{prop nuc A = nuc AT}, $\Nuc(A_v)= D\,\Nuc(A_v^T)$, for some positive diagonal matrix $D$.

 The intersection  $g^{-1}(c)\cap \Escr$  
is described by the non-linear system 
\begin{displaymath}
x\in g^{-1}(c)\cap \Escr\quad\Leftrightarrow\quad \left \{ \begin{array}{ll}
B\log x=c\\
C(x-q)=0
\end{array} \right.\;.
\end{displaymath}
Considering $u=\log x$, this system becomes
\begin{displaymath}
\left \{ \begin{array}{ll}
B\,u=c\\
C(e^u-q)=0 
\end{array} \right.\;.
\end{displaymath}
It is now enough to see that 
\begin{displaymath}
{\left \{ \begin{array}{ll}
B\,u=c\\
C(e^u-q)=0
\end{array} \right.} \quad \textrm{and} \quad
{\left \{ \begin{array}{ll}
B\,u'=c\\
C(e^{u'}-q)=0
\end{array} \right.}
\end{displaymath}
imply  $u=u'$.
By the mean value theorem, for every  $i\in\{1,\dots,n\}$ there is some $\tilde{u}_i\in[u_i,u_i']$ such that
\begin{displaymath}
e^{u_i}-e^{u_i'}=e^{\tilde{u}_i}(u_i-u_i'),
\end{displaymath}
which in vector notation is to say that
$$ e^u-e^{u'}   =   D_{e^{\tilde{u}}}(u-u')  =   e^{\tilde{u}}*(u-u'). $$
Hence
\begin{eqnarray}
{\left \{ \begin{array}{ll}B\,(u-u')=0\\ C(e^u-e^{u'})=0 \end{array} \right.}
\quad \nonumber & \Leftrightarrow & \quad
{\left \{ \begin{array}{ll}	B\,(u-u')=0\\ C\,D_{e^{\tilde{u}}}(u-u')=0 \end{array} \right.} \\ \nonumber
& \Leftrightarrow & \quad
{\left[\begin{array}{c} B \\ \hline C\,D_{e^{\tilde{u}}} \end{array}\right]}(u-u')=0  \\ \nonumber
& \Leftrightarrow & \quad 
 U{\left[\begin{array}{c|c} I & 0 \\ \hline 0 & D_{e^{\tilde{u}}} \end{array}\right]}(u-u')=0 \;. \nonumber
\end{eqnarray}
Therefore, because $\left[\begin{array}{c|c}
					        I & 0 \\
					        \hline
					        0 & D_{e^{\tilde{u}}}
					        \end{array}\right]$ is non-singular, we must have $u=u'$.

Restricting $\Fscr$ to ${\rm int}(\Gamma_\nund)$
we obtain a $X_{\nund,A}$-invariant foliation on ${\rm int}(\Gamma_\nund)$.
Notice that the restriction ${g_\vert}_{{\rm int}(\Gamma_\nund)}:{\rm int}( \Gamma_\nund)\to \Rr^k$ is invariant by the flow of $X_{\nund,A}$ because all its components are.

Since all points in ${\rm int}(\Gamma_\nund)\cap \Escr$ are equilibria,  each leaf of the restricted foliation contains exactly one equilibrium point.
\end{proof}

\begin{defn}
We call attractor\index{polymatrix replicator!attractor} of the polymatrix replicator~\eqref{ode:pmg}
to the following topological closure
 $$\attr{\nund,A} :=\overline{ \cup_{x\in \Gamma_{\nund}} \omega(x) }\;, $$
 where $\omega(x)$ is the $\omega$-limit of $x$ by the  flow
 $\{\varphi_{\nund,A}^t:\Gamma_{\nund} \to\Gamma_{\nund}\}_{t\in\Rr}$. 
\end{defn}

\begin{proposition}\label{prop:La Salle poly}
Given a dissipative polymatrix replicator associated to $(\nund,A)$
with an equilibrium $q\in\inter\left(\Gamma_{\nund}\right)$ 
and a diagonal matrix $D$ as in Definition~\ref{def cons 1}, we have that
$$ \attr{\nund,A}\subseteq \left\{\, x\in \Gamma_\nund \,:\, Q_{D^{-1}A}(x-q)=0\,\right\}\,. $$
\end{proposition}

\begin{proof}
By Theorem~\ref{La Salle theorem}
the attractor $\attr{\nund,A}$ is contained in the
region where $\dot h=0$. The conclusion  follows then by Lemma~\ref{lemma dot h}.
\end{proof}

Given an admissible polymatrix replicator associated to $(\nund,A)$
with an equilibrium $q\in\inter\left(\Gamma_{\nund}\right)$, 
we say that a strategy $i$ {\em is of type $\blV$}  to mean that the following inclusion holds  $\attr{\nund,A}\subseteq \{ x\in\Gamma_\nund \colon  x_i=q_i\}$.
Similarly, we say that a strategy $i$ {\em is of type $\oplus$} to state that 
$\attr{\nund,A}\subseteq \{ x\in\Gamma_\nund \colon  X_{\nund,A}^i(x)=0 \}$, where $X_{\nund,A}^i(x)$ stands for the $i$-th component of the vector $X_{\nund,A}(x)$.
Given two strategies $i$ and $j$ in the same group, we say that {\em  $i$ and $j$ are   related}
when the orbits on the attractor $\attr{\nund,A}$ preserve the foliation
$\{\,\frac{x_i}{x_j}=\text{const.}\,\}$.

For any $v\in V_\nund$ we will denote by $a^v_{ij}$
the entries of the matrix $A_v$.

With this terminology we have

\begin{proposition}
Given an admissible polymatrix game $(\nund,A)$
with an equilibrium $q\in\inter\left(\Gamma_{\nund}\right)$ the following statements hold:
\begin{itemize}
	\item[(1)] For any graph $G(A_v)$ with $v\in V_{\nund,A}^*$:
\begin{itemize}
	\item[(a)] if $i$ is a strategy such that $v_i=0$ and $a_{ii}^v<0$, then  $i$ is of type $\bullet$;
	\item[(b)] if  $j$ is a strategy of type $\bullet$ or $\oplus$  and all neighbours of $j$ but $\ell$ in $G(A_v)$ are of type $\bullet$, then  $\ell$ is of type $\bullet$;
	\item[(c)] if $j$ is a strategy of type $\bullet$ or $\oplus$  and all neighbours of $j$ but $\ell$ in $G(A_v)$ are of type $\bullet$ or $\oplus$, then   $\ell$ is of type $\oplus$;
\end{itemize}
	\item[(2)] For any graph $G(A_v)$ with $v\in V_{\nund}$:
\begin{itemize}
	\item[(d)] if  all neighbours of a strategy $j$
in $G(A_v)$ are of type $\bullet$ or $\oplus$, then $j$ is related to  the unique strategy $j'$, in the same group as $j$, such that $v_{j'}=1$.
\end{itemize}
\end{itemize}
\end{proposition}

\begin{proof}
The proof involves the manipulation of algebraic relations holding on the attractor.
To simplify the terminology we will say that some algebraic relation holds to mean that it holds on the attractor.

Choose a positive diagonal matrix $D$ of type $\nund$ such that $Q_{AD}\leq 0$ on $H_\nund$,
and set $\tilde A := D^{-1} A$.
By Lemma~\ref{A sd and D positive implies AD sd},
for any $v\in V_\nund$, the matrices 
$A_v$ and $\tilde A_v$ have the same dissipative and stably dissipative character.
Hence $V_{\nund,A}^\ast= V_{\nund,\tilde A}^\ast$.

Given $v\in V_{\nund,A}^\ast$, for any solution $x(t)$ of the polymatrix replicator in the attractor, we have that $Q_{\tilde A_v}\left(x(t)-q\right)=0$.
Hence, as $\tilde A_v$ is stably dissipative and $a_{ii}^v<0$,
by Lemma~\ref{lema2.1.ZL2010} follows that $x_i(t)=q_i$ on the attractor,
which proves $(a)$.

Given $v\in V_{\nund,A}^*$ we have that $\tilde A_v$ is stably dissipative.
By Proposition~\ref{prop:La Salle poly}, we obtain 
$$\sum_{(k,\ell)\in \Vscr_v} \tilde A_{(j,j'),(k,\ell)}(x_k-q_k)=0 $$
on the attractor.

Observe that if $j$ is of type $\bullet$, then $x_j=q_j$, and if
$j$ is of type $\oplus$, then $a^v_{jj}=A_{(j,j'),(j,j')}=0$,
where $j'$ is the unique strategy  in the same group as $j$  such that $v_{j'}=1$.

Let $j,\ell$ be neighbour vertices in the graph $G(A_v)$.

Let us prove $(b)$.
If $j$ is of type $\bullet$ or $\oplus$ and all of its neighbours are of type $\bullet$, except for  $\ell$, then
$$ \tilde A_{(j,j'),(\ell,\ell')}(x_\ell-q_\ell)=0\,, $$
from which follows that $x_\ell=q_\ell$ because
$A_{(j,j'),(\ell,\ell')} = d_j \tilde A_{(j,j'),(\ell,\ell')} \neq 0\,$, where $\ell'$ is the unique strategy  in the same group as $\ell$  such that $v_{\ell'}=1$.
This proves $(b)$.

Let us prove $(c)$.
If $j$ is of type $\bullet$ or $\oplus$ and all of its neighbours are of type $\bullet$ or $\oplus$, except for  $\ell$, then
$$ A_{(j,j'),(\ell,\ell')}(x_\ell-q_\ell)=c\,, $$
for some constant $c$. Hence  because
$A_{(j,j'),(\ell,\ell')}\neq 0$, $x_\ell$ is constant  which proves $(c)$.

Let us prove $(d)$. Suppose all  neighbours of a strategy $j$
are of type $\bullet$ or $\oplus$. By the  polymatrix quotient
rule (see~Proposition~\ref{prop:quotient rule}),
$$ 
\frac{d}{dt}\left(\frac{x_j}{x_{j'}}\right)=
\frac{x_j}{x_{j'}}\, \sum_{(k,\ell)\in \Vscr_v}A_{(j,j'),(k,\ell)}\, (x_k-q_k) \;.
$$

Since all neighbours of $j$ are of type $\bullet$ or $\oplus$ we obtain
$$ \frac{d}{dt}\left(\frac{x_j}{x_{j'}}\right)= \frac{x_j}{x_{j'}}\, C \;, $$
for some constant $C$. Hence
$$ \frac{x_j}{x_{j'}}=B_0\,e^{Ct}  \;, $$
where $B_0=\frac{x_j(0)}{x_{j'}(0)}$. By Proposition~\ref{prop:limit quot rule}
we have that the constant $C$ must be $0$.
Hence there exists a constant $B_0>0$ such that $\frac{x_j}{x_{j'}}=B_0$, which proves $(d)$. 
\end{proof}

\begin{proposition}
If in a group $\alpha$  all strategies  are of type $\bullet$ (respectively of type $\bullet$ or $\oplus$) except for one strategy $i$, then $i$ is of type $\bullet$ (respectively of type $\oplus$).
\end{proposition}

\begin{proof}
Suppose that in a group $\alpha$ all strategies are
of type  $\bullet$ or $\oplus$ except for one strategy $i$.
We have that $x_k=c_k$, for some constant $c_k$, for each $k\neq i$.
Thus, 
$$x_i=1-\sum_{\substack{j\in\alpha \\ j\neq i}} x_j=1-\sum_{j=\bullet} x_j-\sum_{k=\oplus} x_k=1-\sum_{j=\bullet} q_j-\sum_{k=\oplus} c_k\,.$$
Hence $i$ is of type $\oplus$.

If in a group $\alpha$ all strategies are
of type  $\bullet$, the proof is analogous.
\end{proof}

\begin{proposition}
Assume that in a group $\alpha$  with $n$ strategies, $n-k$ of them, with $0\le k<n$,
are of type $\bullet$  or $\oplus$, and denote by $S$ the set of the remaining $k$ strategies.
If the graph with vertex set $S$, obtained drawing an edge between every pair of related strategies in $S$, is connected, then all strategies in $S$ are of type $\oplus$.
\end{proposition}

\begin{proof}
Since all strategies in $\,\alpha\setminus S\,$ are of type $\bullet$  or
$\oplus$, for the strategies in $S$ we have that
\begin{equation}\label{alpha_minus_S}
\sum_{i\in S} x_i=1-C\,,
\end{equation}
where $C=\sum_{j\in\alpha\setminus S} x_j$.

Let $G_S$ be the graph with vertex set $S$ obtained drawing an edge between every pair of related strategies in $S$. 
Since $G_S$ is connected we have that it contains a tree.
Considering the $k-1$ relations between the strategies in $S$ given by that tree,
we have $k-1$ linearly independent equations of the form $x_i=C_{ij}x_j$ for pairs of
strategies $i$ and $j$ in $S$, where $C_{ij}$ is a constant. Together with~\eqref{alpha_minus_S} we obtain $k$ linear independent equations for the $k$ strategies in $S$, which implies that $x_i=\text{constant}$, for every $i\in S$. This concludes the proof.
\end{proof}

Based on these facts we introduce a reduction algorithm\index{reduction algorithm!polymatrix} on the set of graphs
$ \{\,G(A_v) \,:\, v\in V_{\nund}\, \} $ to derive information on the strategies
of an admissible polymatrix game $(\nund,A)$.

In each step, we also register the information obtained about each strategy in what we call the ``information set'', where all strategies of the polymatrix are represented.

The algorithm is about labelling (or colouring)  strategies with the \linebreak
 ``colours'' $\bullet$ and $\oplus$.
The algorithm acts upon all graphs $G(A_v)$ with $v\in V_{\nund}$ as well as on the
information set.
It is implicit that after each rule application, the
new labels (or colours) are transferred between the graphs $G(A_v)$ and the information set,
that is, if in a graph $G(A_v)$ a strategy $i$ has been coloured
$i=\bullet$, then in all other graphs containing the strategy $i$,
we colour it  $i= \bullet$, as well on the information set.

Some rules just can be applied to graphs $G(A_v)$ such that $v\in V_{\nund,A}^*$,
while others can be applied to all graphs.

\begin{PolyRule}\label{Rule_a}
Initially, for each graph $G(A_v)$ such that $v\in V_{\nund,A}^*$ colour in black ($\bullet$) any strategy $i$  such that $a_{ii}^v<0$.
Colour in white ($\circ$) all other strategies.
\end{PolyRule}

The reduction procedure consists in applying the following rules, corresponding to valid inferences rules. For each graph $G(A_v)$ such that $v\in V_{\nund,A}^*$:

\begin{PolyRule}\label{Rule_b}
If  $i$ has colour $\bullet$ or $\oplus$ and all neighbours of $i$ but $j$ in $G(A_v)$ are $\bullet$, then   colour $j=\bullet$.
\end{PolyRule}

\begin{PolyRule}\label{Rule_c}
If $i$ has colour $\bullet$ or $\oplus$ and all neighbours of $i$ but $j$ in $G(A_v)$ are $\bullet$ or $\oplus$, then   colour $j=\oplus$.
\end{PolyRule}

For each graph $G(A_v)$ such that $v\in V_{\nund}$:

\begin{PolyRule}\label{Rule_d}
If $i$ has colour $\circ$ and all neighbours of $i$ in $G(A_v)$ are $\bullet$ or $\oplus$, then we put a link between strategies $j$ and $j'$ in the ``information set'', where $j'$ is the unique strategy such that $v_{j'}=1$ and $j'$ is in the same group as $j$.
\end{PolyRule}

The following rules can be applied to the set of all strategies of the polymatrix game.

\begin{PolyRule}\label{Rule_e}
If in a group all strategies have colour $\bullet$ (respectively, $\bullet , \oplus$) except for one strategy $i$, then colour $i=\bullet$ (respectively, $i=\oplus$).
\end{PolyRule}

\begin{PolyRule}\label{Rule_f}
If in a group some strategies have colour $\bullet$  or $\oplus$, 
and the remaining strategies are related forming a connected graph,
then colour with $\oplus$ all that remaining strategies.
\end{PolyRule}

We define the  {\em reduced information set}   $\Rscr(\nund,A)$
as the $\{\bullet ,  \oplus,  \circ\}$-coloring on the set of strategies $\{1,\ldots, n\}$, which is obtained    by successive applications to the graphs $G(A_v)$, $v\in V_\nund$, of the reduction rules  1-6,  until they can no longer be applied.

\begin{proposition}
\label{prop poly reduced}
Let $(\nund,A)$ be an admissible polymatrix game,
and consider the associated polymatrix replicator~\eqref{ode:pmg}
with an interior equilibrium
$q\in {\rm int}(\Gamma_\nund)$.

\begin{enumerate}
\item If all vertices of $\Rscr(\nund,A)$ are $\bullet$ then 
$q$ is the unique globally attractive equilibrium.
\item If $\Rscr(\nund,A)$ has only $\bullet$ or $\oplus$ vertices then 
there exists an invariant foliation with a unique globally attractive equilibrium in each leaf.
\end{enumerate}
\end{proposition}

\begin{proof}
Item (1) is clear because if all strategies are of type $\bullet$
then for every orbit $x(t)=(x_1(t),\ldots, x_n(t))$ of ~\eqref{ode:pmg}, and every $i=1,\ldots, n$, one has\, 
$\lim_{t\to +\infty} x_i(t)=q_i$.

Likewise, if $\Rscr(\nund,A)$ has only $\bullet$ or $\oplus$ vertices
then every orbit of ~\eqref{ode:pmg} converges to an equilibrium point, which depends on the initial condition.
But by Proposition~\ref{prop invariant foliation} there exists an
invariant foliation $\Fscr$ with a single equilibrium point in each leaf. Hence, the unique equilibrium point in each leaf of $\Fscr$ must be globally attractive. 
\end{proof}

The following definition corresponds to a one-step reduction
of the attractor dynamics.

\begin{defn}\label{nl,Al}
Given a polymatrix game $(\nund,A)$, a strategy $\ell\in \alpha$, for some group $\alpha$,
and a point $q\in\inter\left(\Gamma_{\nund}\right)$, we call
{\em $(q,\ell)$-reduction of $(\nund,A)$} to a new polymatrix game $(\nund(\ell),A(\ell))$ obtained removing the strategy $\ell$ from the group $\alpha$, where
$\nund (\ell):=(n_1,\dots, n_{\alpha-1}, n_{\alpha}-1,n_{\alpha+1},\dots ,n_p)$,
and  the matrix $A(\ell)=(a_{ij}(\ell))$ indexed in 
$\{1,\dots, \ell-1,\ell+1,\dots, n\}$ has the following entries:
\begin{equation}\label{entries of A(l)}
a_{ij}(\ell):=\left \{ \begin{array}{ll}
a_{ij}-a_{lj}  & \text{ if } \; j\notin \alpha\\
(a_{ij}-a_{lj})(1-q_\ell)+(a_{il}-a_{ll})q_\ell  & \text{ if } \; j\in \alpha\setminus\{\ell\}\;.
\end{array} \right.
\end{equation}
\end{defn}

The map
$\psi_\ell:\Gamma_{\nund}\cap\{x_\ell=q_\ell\} \to \Gamma_{\nund(\ell)}$,
  $\psi_\ell(x)=\xred{\ell}=(x_j)_{j\neq \ell}$, defines  a natural  identification.

\begin{proposition}\label{Reduced_Poly_1}
Let $(\nund,A)$ be a polymatrix game with an equilibrium
$q\in\inter\left(\Gamma_{\nund}\right)$.
Given a  strategy $\ell\in \alpha$, for some group $\alpha$, 
the $(q,\ell)$-reduction $(\nund(\ell),A(\ell))$ of $(\nund,A)$ is such that if $x\in \Gamma_{\nund}\cap \{x_\ell=q_\ell\}$ and
$X_{\nund,A}(x)$ is tangent to $\{x_\ell=q_\ell\}$, that is $X_{\nund,A}^\ell(x)=0$, then for all $j\neq \ell$,
$$ X_{\nund,A}^j(x) = X_{\nund(\ell), A(\ell)}^j(\xred{\ell})\;.$$
\end{proposition}

\begin{proof}
Suppose that for some $\alpha\in\{1,\dots,p\}$ there exists $\ell\in\alpha$ such that
$x\in \Gamma_{\nund}\cap \{x_\ell=q_\ell\}$ and $X_{\nund,A}^\ell(x)=0$.

Since $\sum_{\substack{j\in\alpha \\ j\neq \ell}}x_j=1-q_\ell$, considering the change of variables
\begin{equation}\label{change_var}
y_j=\left\{
\begin{array}{ll}
\frac{x_j}{1-q_\ell}  \quad & \text{ if } \; j\in\alpha\setminus \{\ell\}\\
x_j  & \text{ if } \; j\notin \alpha 
\end{array}
\right.\;,
\end{equation}
we have that $\sum_{j\in\alpha\setminus \{\ell\}}y_j=1$\,.

By Proposition~\ref{equivalent matrices prop}, we can assume $A=(a_{ij})$ has all entries equal to zero in row $\ell$, i.e., $a_{lj}=0$ for all $j$. Thus we obtain 
\begin{equation}\nonumber
\frac{dx_\ell}{dt}=x_\ell\left( -\sum_{\beta=1}^p (x^{\alpha})^ t A^{\alpha\beta} x^{\beta}\right)\,.
\end{equation}

Hence, making $x_\ell=q_\ell$,  the replicator equation~\eqref{ode:pmg} becomes
\vspace{-2.5mm}
\begin{itemize}\setlength\itemsep{-1mm}
	\item[(i)] if $i\in\alpha\setminus \{\ell\}$, 
	\begin{eqnarray}\label{rep_eq_alpha}
	\frac{dx_i}{dt}=x_i\left( \sum_{\substack{j=1\\ j\neq \ell}}^na_{ij}x_j+a_{il}q_\ell-
	\sum_{\substack{k\in\alpha\\ k\neq \ell}}\sum_{j=1}^n a_{kj}x_kx_j \right)
	\end{eqnarray}
	
	\item[(ii)] if $i\in\beta\neq\alpha$, the equation is essentially the same, with $x_\ell=q_\ell$. 
\end{itemize}

Observe that $\sum_{\beta=1}^p (x^{\alpha})^ t A^{\alpha\beta} x^{\beta}=0$ because we are assuming that $x\in \Gamma_{\nund}\cap \{x_\ell=q_\ell\}$ and
$X_{\nund,A}^\ell(x)=0$.  

Hence we can add 
  $$-\frac{q_\ell}{1-q_\ell}\sum_{\beta=1}^p (x^{\alpha})^ t A^{\alpha\beta} x^{\beta}$$ to each equation for $\frac{dx_i}{dt}$, with $i\in\alpha\setminus \{\ell\}$, without changing the
vector field $X_{\nund,A}$ at the points $x\in \Gamma_{\nund}\cap \{x_\ell=q_\ell\}$ where 
$X_{\nund,A}(x)$ is tangent to $\{ x_\ell=q_\ell\}$.
So equation ~\eqref{rep_eq_alpha} becomes
\begin{equation}\label{rep_eq_alpha_2}
	\frac{dx_i}{dt}=x_i\left( \sum_{\substack{j=1\\ j\neq \ell}}^na_{ij} x_j+a_{il}q_\ell-
	\frac{1}{1-q_\ell}\sum_{\substack{k\in\alpha\\ k\neq \ell}}\sum_{j=1}^n a_{kj}x_kx_j \right)
\end{equation}

Now, using the change of variables~\eqref{change_var}, equation~\eqref{rep_eq_alpha_2} becomes
\begin{equation}\label{rep_eq_alpha_3}
\frac{dy_i}{dt}=y_i\left( f_i-\sum_{\substack{k\in\alpha\\ k\neq \ell}}y_kf_k \right) \quad (i\in \alpha) \;,
\end{equation}
where $f_i=\sum_{j\in\alpha\setminus \{\ell\}} a_{ij}(1-q_\ell)y_j+a_{il}q_\ell
			+\sum_{j\notin\alpha} a_{ij} y_j$.

Let $\check{\alpha}\equiv \alpha\setminus\{\ell\}$.
Setting $a_{il}q_\ell=a_{il}q_\ell(\sum_{j\in\check{\alpha}} y_j)$,
\begin{equation}\label{rep_eq_alpha_4}
\frac{dy_i}{dt}=y_i\left( g_i-\sum_{k\in\beta} y_k g_k \right), \quad i\in\beta,\; \beta\in\{1,\dots,p\} \;,
\end{equation}
where 
$g_i=\sum_{j\in\check{\alpha}}(a_{ij}(1-q_\ell)+a_{il}q_\ell)y_j+ \sum_{j\notin\check{\alpha}}a_{ij} y_j$,
defines a new polymatrix game in dimension $n-1$.
In fact,~\eqref{rep_eq_alpha_4} is the replicator equation of the polymatrix game $(\nund(\ell),A(\ell))$, where, since we have assumed that $a_{lj}=0$ for all $j$,~\eqref{entries of A(l)} becomes
\begin{equation}\nonumber
a_{ij}(\ell)=\left \{ \begin{array}{ll}
a_{ij}  & \text{ if } \; j\notin\check{\alpha}\\
a_{ij}(1-q_\ell)+a_{il}q_\ell  & \text{ if } \; j\in\check{\alpha} \;.
\end{array} \right.
\end{equation}
\end{proof}

\begin{remark}\label{group of cardinal 2}
Under the assumptions of Proposition~\ref{Reduced_Poly_1},
when  $n_\alpha=2$, considering for instance that the group $\alpha$ consists of strategies $\ell-1$ and $\ell$,  $x_\ell=q_\ell$ implies that $x_{\ell-1 }=1-q_\ell=q_{\ell-1}$.
Hence we can further reduce the polymatrix game  $(\nund(\ell),A(\ell))$ to a new polymatrix game 
with  type $(n_1,\dots, n_{\alpha-1}, n_{\alpha+1},\dots ,n_p)$ 
and payoff matrix  
indexed in $\{1,\dots, \ell-2,\ell+1,\dots, n\}$.
\end{remark}

\begin{corollary}
Let $(\nund,A)$ be a polymatrix game with an equilibrium
$q\in\inter\left(\Gamma_{\nund}\right)$.
Given a  set $Q\subset \{1,\ldots, n\}$ of strategies
such that 
$$ \attr{\nund,A} \subseteq \bigcap_{\ell\in Q} \{x_\ell=q_\ell\}\;, $$
then  there exists a new polymatrix game $(\mund,B)$, where $m_\alpha=\abs{\alpha\setminus Q}$ for every $\alpha=1,\ldots, p$,  and an identification $\psi:\Gamma_{\nund}\cap \bigcap_{\ell\in Q} \{x_\ell=q_\ell\} \to \Gamma_{\mund}$  such that $X_{\nund,A}=X_{\mund,B}\circ \psi$ on the attractor $\attr{\nund,A}$.

In other words, the attractor $\attr{\nund,A}$ lives on a lower dimension polymatrix replicator of type $\mund$.
\end{corollary}

\begin{proof}
Apply Proposition~\ref{Reduced_Poly_1} repeatedly.
\end{proof}

\begin{lemma}\label{diag property}
Given a polymatrix game $(\nund,A)$ and a diagonal matrix $D$ of type $\nund$, we have
$$ (A\,D)_v = A_v\, D_v\;, $$
where    $A_v$ is given in Definition~\ref{def Av} and
$D_v$ is the submatrix of $D$ indexed in $\Vscr_v = \{\, i\in \{1,\ldots, n\}\,\colon\, v_i=0\,\}$.
\end{lemma}

\begin{proof}
Given indices $i,k\in \Vscr_v$, take $j$, resp. $\ell$,
in the group of $i$, resp. $k$, such that $v_j=v_\ell=1$.

Since $D$ is of type $\nund$ we have $d_k=d_\ell$.
By Definition~\ref{def Av},  
\begin{align*}
((A\,D)_v)_{ik} = (A\,D)_{(i,j),(k,\ell)} & =
a_{ik} d_k + a_{j\ell} d_\ell - a_{i\ell} d_\ell - a_{jk} d_k\\
 & = ( a_{ik}  + a_{j\ell}   - a_{i\ell}   - a_{jk}  )\, d_k\\
 &= A_{(i,j),(k,\ell)}\,d_k = (A_v \,D_v)_{ik}\;.
\end{align*}
\end{proof}

\begin{lemma}\label{Reduced_Poly_2}
Let $(\nund,A)$ be an admissible polymatrix game and $D$ a  diagonal matrix as in Definition~\ref{defi:dissipative_poly}. Given $v\in V_{\nund,A}^*$ such that
$v_\ell=0$ and $a_{ll}^v<0$ for some $\ell\in\alpha$ with $\alpha\in\{1,\dots,p\}$, there exists a positive diagonal matrix $\check{D}$ of type $\nund(\ell)$ such that
$(A(\ell)\check{D})_{\check{v}}$ is the submatrix of $(AD)_v$ obtained eliminating
row and column $\ell$. Moreover
\begin{itemize}
	\item[(a)] $(\nund(\ell),A(\ell))$ is admissible, and;
	\item[(b)] $\check{v}\in V_{\nund(\ell),A(\ell)}^*$.
\end{itemize}
\end{lemma}

\begin{proof}
By Proposition~\ref{equivalent matrices prop}, we can assume $A=(a_{ij})$ has all entries equal to zero in row $\ell$, i.e., $a_{lj}=0$ for all $j$.

Since $(\nund,A)$ is admissible and $v\in V_{\nund,A}^\ast$, $(AD)_v$ is stably dissipative.

Consider the set
 $I=\{i\in\{1,\dots,n\}\,:\, v_i=0 \; \text{ and } \; a^v_{ii}=0\,\}$.
By Proposition~\ref{sd:char},  the submatrix $B_v=(a^v_{ij}\,d_j)_{i,j\in I}$ of $(AD)_v=A_v D_v$  is skew-symmetric.

Let $\Gamma_{\nund(\ell)}$ be the polytope corresponding to the new polymatrix replicator in lower dimension, given by Proposition~\ref{Reduced_Poly_1} and defined by matrix $A(\ell)=\left(a_{ij}(\ell)\right)_{i,j\neq \ell}$.

Observing that $v_i=0$ for all strategies $i$ of the matrix $(AD)_v$, we can choose the vertex $\check{v}$ in the
polytope $\Gamma_{\nund(\ell)}$ determined by the exact same strategies as $v$.
Notice that $v_\ell=0$ for the removed strategy $\ell$.

As   in the proof of Proposition~\ref{Reduced_Poly_1}  the
matrix $A(\ell)$ is defined by
\begin{equation}\nonumber
a_{ij}(\ell)=\left \{ \begin{array}{ll}
a_{ij}  & \text{ if } \; j\notin\check{\alpha}\\
a_{ij}(1-q_\ell)+a_{il}q_\ell  & \text{ if } \; j\in\check{\alpha} \;.
\end{array} \right.
\end{equation}
Hence
\begin{equation}\nonumber
a_{ij}^{\check{v}}(\ell)=\left \{ \begin{array}{ll}
a_{ij}^v  & \text{ if } \; j\notin\check{\alpha}\\
(1-q_\ell)a_{ij}^v  & \text{ if } \; j\in\check{\alpha} \;,
\end{array} \right.
\end{equation}
where $a_{ij}^{\check{v}}(\ell)\equiv (a_{ij}(\ell))^{\check{v}}$ are the entries
of matrix $A(\ell)_{\check{v}}$.

Considering the positive diagonal matrix
$$\check{D}=\textrm{diag}\left(I_1,\dots ,\frac{1}{1-q_\ell}I_\alpha,\dots ,I_p\right)\,,$$
we have that $(A(\ell)\check{D})_{\check{v}}$ is the submatrix $B_v$ of
$(AD)_v$ obtained by removing the row and column corresponding to strategy $\ell$.
By Lemma~\ref{diag property}, $(A(\ell)\check{D})_{\check{v}} = A(\ell)_{\check{v}}\, \check{D}_{\check{v}}$.
Hence, by Lemma~\ref{lemma 3.1.5. masterthesis}, $A(\ell)_{\check{v}}\, \check{D}_{\check{v}}$
is stably dissipative, and consequently, by
Lemma~\ref{A sd and D positive implies AD sd}, $A(\ell)_{\check{v}}$ is also
stably dissipative.
\end{proof}

Proposition~\ref{Reduced_Poly_1} and Lemma~\ref{Reduced_Poly_2} allows us to generalize~\cite[Theorem 4.5]{DFO1998}  about the Hamiltonian nature
of the limit dynamics in {\em admissible} polymatrix replicators.

\begin{theorem}\label{theor:Reduced_Poly_Conserv}
Consider a polymatrix replicator~\eqref{ode:pmg}  on $\Gamma_{\nund}$,
and assume that the system is admissible and has an equilibrium
$q\in\inter\left(\Gamma_{\nund}\right)$.
Then the limit dynamics of~\eqref{ode:pmg} on the attractor
$\attr{\nund,A}$ is described by a Hamiltonian polymatrix replicator in some lower dimensional prism $\Gamma_{\nund'}$.
\end{theorem}

\begin{proof}
By definition there exists a vertex $v\in\Gamma_{\nund}$ such that $A_v=(a_{ij}^v)$ is stably dissipative. Applying Proposition~\ref{Reduced_Poly_1} and Lemma~\ref{Reduced_Poly_2} we obtain a new polymatrix replicator in lower dimension that is admissible.

We can iterate this process until the corresponding vertex $\check{v}$ in the polytope  is such that,
$a_{ii}^{\check{v}}=0$ for all $i$ with $\check{v}_i=0$.

Let us denote  the resulting polymatrix game by $(\rund,A')$.
By Proposition~\ref{sd:char}, for some positive diagonal matrix
$D'$ of type $\rund$,  $(A'\,D')_{\check{v}}$ is skew-symmetric.
Hence $Q_{A'\,D'}=0$ on $H_{\rund}$, and by Definition~\ref{def cons 1} the polymatrix game $(\rund, A')$ is conservative. Notice that this polymatrix game has essentially the same formal equilibrium up to coordinate rescalings.
Thus by Proposition~\ref{Conserv_Poly} the vector field
$X_{\rund,A'}$ is Hamiltonian.
\end{proof}


\section{An Example}
\label{applications}

Consider the polymatrix replicator system associated to 
the polymatrix game  $\Gcal=\left((3,2),A\right)$, where
$$A=\left[\begin{array}{ccccc}
			-1 & 8 & -7 & 3 & -3 \\
			-10 & -1 & 11 & 3 & -3 \\
			11 & -7 & -4 & -6 & 6 \\
			-3 & -3 & 6 & 0 & 0 \\
			3 & 3 & -6 & 0 & 0 \\
			  \end{array}\right]\,.$$

We denote by $X_\Gcal$ the vector field associated to this polymatrix replicator defined on the popytope $\Gamma_{(3,2)}=\Delta^2\times\Delta^1\,.$

\begin{figure}[h]
\begin{center}
\includefigure{width=6cm}{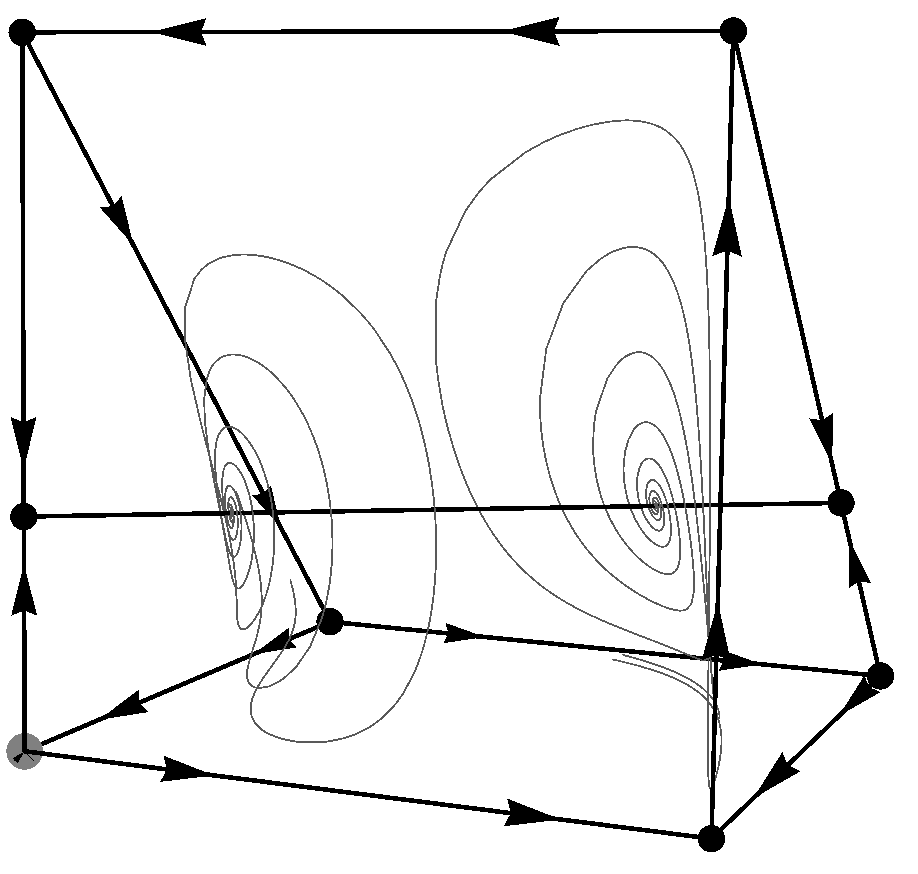}
\caption{\footnotesize{Four orbits in two different leafs of the polymatrix game $\Gcal$.
}}\label{two_leafs}
\end{center} 
\end{figure}

In this example we want to illustrate the reduction algorithm on the set of graphs
$ \{\,G(A_v) \,:\, v\in V_{(3,2)}\, \} $ to derive information on the strategies
of the polymatrix game $\Gcal$ as described in section~\ref{dissipative}.
We will see that this polymatrix game is admissible and 
verify the validity of the conclusion of Theorem~\ref{theor:Reduced_Poly_Conserv} for this example.

\begin{table}[h]
\centering
\begin{tabular}[c]{cccccc}
\\
\hline
\\[-3mm]
$v_1=(1,4)$ & $v_2=(1,5)$ & $v_3=(2,4)$ & $v_4=(2,5)$ & $v_5=(3,4)$ & $v_6=(3,5)$ \vspace{1mm} \\
\hline
\end{tabular}
\caption{\footnotesize{Vertex labels.}} \label{ex:vert_0}
\end{table}

In this game the strategies are divided in two groups, $\{1,2,3\}$ and $\{4,5\}$.
The vertices of the  phase space  $\Gamma_{(3,2)}$
will be designated by pairs in
$\{1,2,3\}\times\{4,5\}$, where the label $(i,j)$ stands for the point
$e_i + e_j\in \Gamma_{(3,2)}$.
To simplify the notation we designate the prism vertices by the letters $v_1,\ldots, v_6$
according to table~\ref{ex:vert_0}.

\begin{table}[t]
\centering
\begin{tabular}{ccc}
Vertex & $A_v$  & $G(A_v)$ \\
\hline \hline \\[-2mm]
\,$v_1\in V_{\nund,A}^*$  &
$\begin{bmatrix} 0 & 27 & 0 \\ -27 & -9 & 18 \\ 0 & -18 & 0 \end{bmatrix}$  &
\includefigure{width=4cm}{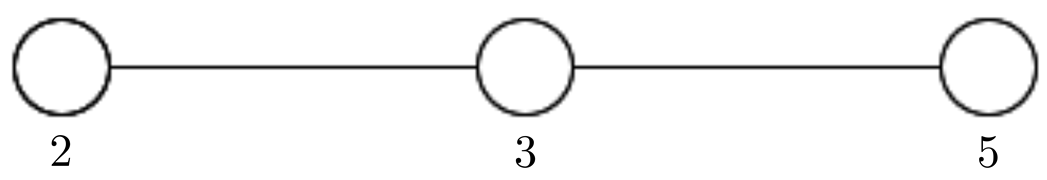} \\[8mm]
\,$v_2\in V_{\nund,A}^*$     &
$\begin{bmatrix} 0 & 27 & 0 \\ -27 & -9 & -18 \\ 0 & 18 & 0 \end{bmatrix}$  &
\includefigure{width=4cm}{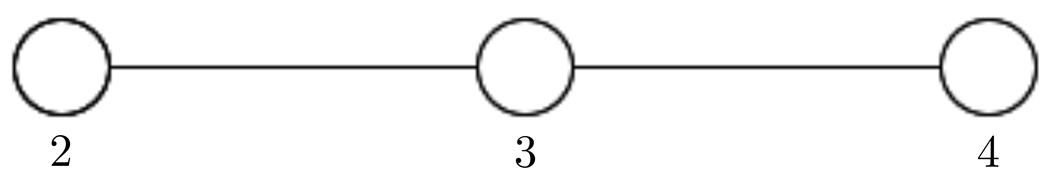} \\[8mm]
\,$v_3\in V_{\nund,A}^*$     &
$\begin{bmatrix} 0 & -27 & 0 \\ 27 & -9 & 18 \\ 0 & -18 & 0 \end{bmatrix}$  &
\includefigure{width=4cm}{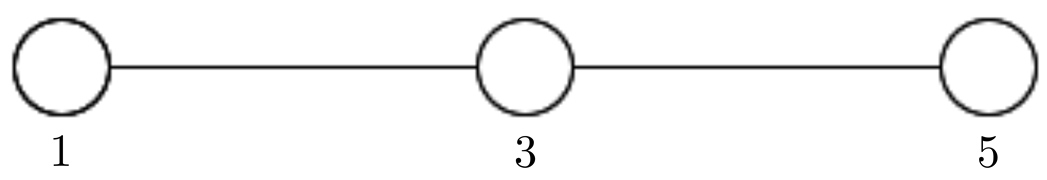} \\[8mm]
\,$v_4\in V_{\nund,A}^*$     &
$\begin{bmatrix} 0 & -27 & 0 \\ 27 & -9 & -18 \\ 0 & 18 & 0 \end{bmatrix}$  &
\includefigure{width=4cm}{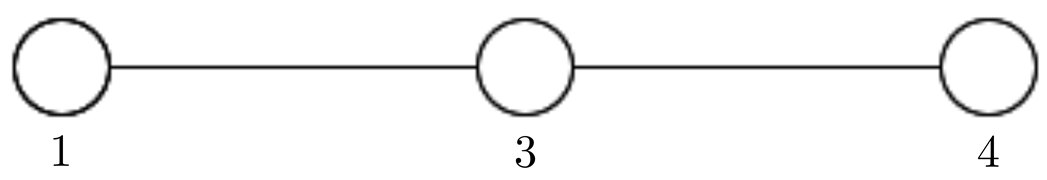} \\[8mm]
 &  &
\multirow{3}{*}{\includefigure{width=2.7cm}{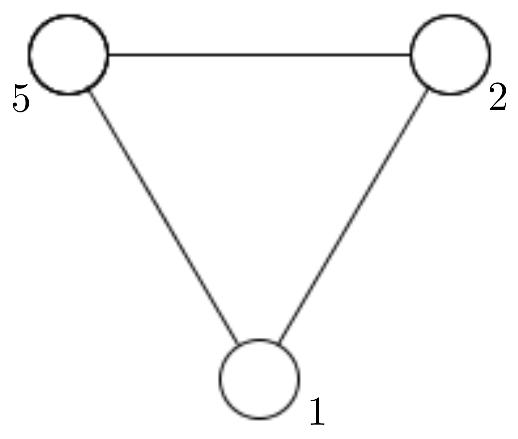}}  \\
\,$v_5\notin V_{\nund,A}^*$ & $\begin{bmatrix} -9 & 18 & -18 \\ -36 & -9 & -18 \\ 18 & 18 & 0 \end{bmatrix}$ & \\
&  &  \\[4mm]
 &  &
\multirow{3}{*}{\includefigure{width=2.7cm}{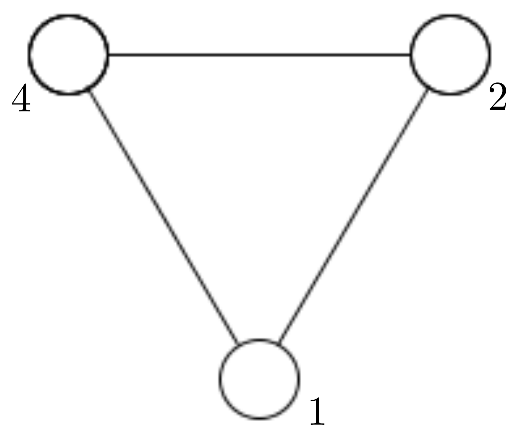}}
\\
\,$v_6\notin V_{\nund,A}^*$  &
$\begin{bmatrix} -9 & 18 & 18 \\ -36 & -9 & 18 \\ -18 & -18 & 0 \end{bmatrix}$ & \\
&  &  \\[1mm]
\hline \hline
\end{tabular}
\caption{\footnotesize{Matrix $A_v$ and its graph $G(A_v)$ for each vertex $v$.}}
\label{ex:A_v_matrices_graphs}
\end{table}

\begin{table}[h] 
\centering
\begin{tabular}{ccccccc}
Step & Rule & Vertex & Strategy & Group 1 & & Group 2 \\
\hline \hline \\[1mm]
$1$ & $1$ & $v_1, v_2, v_3, v_4$ & $3$ &
\includefigure{width=2.2cm}{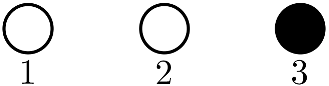} &\,&
\includefigure{width=1.3cm}{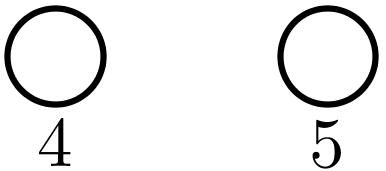} \\[4mm]
\hline\\[1mm]
$2$ & $4$ & $v_4$ (or $v_5$) & $4, 5$ &
\includefigure{width=2.2cm}{Graph_Inf_Set_1_2} &\,&
\includefigure{width=1.3cm}{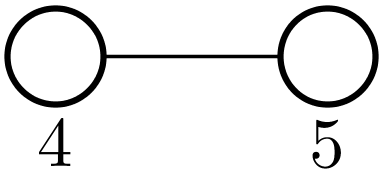} \\[4mm]
\hline\\[1mm]
$3$ & $6$ & $-$ & $4, 5$ &
\includefigure{width=2.2cm}{Graph_Inf_Set_1_2} &\,&
\includefigure{width=1.3cm}{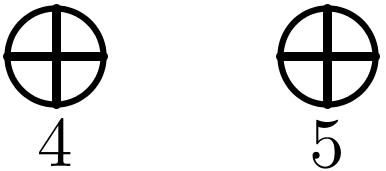} \\[4mm]
\hline\\[1mm]
$4$ & $3$ & $v_1, v_2$ & $1, 2$ &
\includefigure{width=2.2cm}{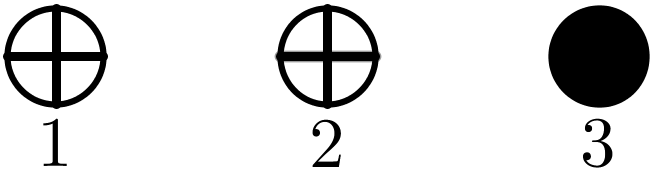} &\,&
\includefigure{width=1.3cm}{Graph_Inf_Set_6_2} \\[4mm]
\hline \hline
\end{tabular}
\caption{\footnotesize{Information set of all strategies (by group) of $\Gcal$,
where for each step, we mention the rule, the vertex (or vertices) and the strategy (or strategies) to which we apply the rule.}}
\label{ex:Information_set}
\end{table}

\newpage

The point $q\in\inter\left(\Gamma_{(3,2)}\right)$ given by
$$ q= \left( \frac{1}{3}, \frac{1}{3}, \frac{1}{3}, \frac{1}{2}, \frac{1}{2}\right)\,,$$
is an equilibrium of our polymatrix replicator $X_\Gcal$. In particular it is also a formal equilibrium of $\Gcal$ (see Definition~\ref{defi:formal_equilibrium}).

The quadratic form $Q_A:H_{(3,2)}\to\Rr$ induced by matrix $A$ is
$$ Q_A(x)=- 9\,x_3^2\,, $$
where $x=(x_1,x_2,x_3,x_4,x_5)\in H_{(3,2)}$.
By Definition~\ref{defi:dissipative_poly}, $\Gcal$ is dissipative.

In table~\ref{ex:A_v_matrices_graphs} we present for each vertex $v$ in the prism the corresponding matrix $A_v$ and graph $G(A_v)$.

Considering vertex $v_1=(1,4)$ for instance, by Proposition~\ref{sd:char},
we have that matrix $A_{v_1}$ is stably dissipative.
Hence, by Definition~\ref{defi:admissible_poly}, $\Gcal$ is admissible and
$v_1\in V_{\nund,A}^*$.

Table~\ref{ex:Information_set} represents the steps of the reduction procedure applied to $\Gcal$. Let us describe it step by step:
\begin{itemize}\setlength\itemsep{-1mm}
	\item[(Step 1)] Initially, considering the vertices $v_1, v_2, v_3$ and $v_4$ we apply
{\em rule~\ref{Rule_a}} to the corresponding graphs $G(A_{v_1}), G(A_{v_2}), G(A_{v_3})$ and
$G(A_{v_4})$, and we colour in black $(\bullet)$ strategy $3$. We obtain the graphs
depicted in column ``Step 1'' in table~\ref{table:reduced_graphs};
	\item[(Step 2)] In this step we can consider vertex $v_4$ (or $v_5$) to apply
{\em rule~\ref{Rule_d}}. Hence, we put a link between strategies $4$ and $5$ in group $2$;
	\item[(Step 3)] In this step we apply {\em rule~\ref{Rule_f}} to strategies $4$ and $5$, and we colour with $\oplus$ that strategies. We obtain the graphs
depicted in column ``Step 3'' in table~\ref{table:reduced_graphs};
	\item[(Step 4)] Finally, we apply {\em rule~\ref{Rule_c}} to vertices $v_2$ and $v_3$ in the corresponding graphs of the column ``Step 3'' in table~\ref{table:reduced_graphs}, and we colour with $\oplus$ the strategy $2$. Analogously  we apply {\em rule~\ref{Rule_c}} to vertices $v_1$ and $v_3$ in the corresponding graphs of the column ``Step 3'' in
table~\ref{table:reduced_graphs}, and we colour with $\oplus$ the strategy $1$.
We obtain the graphs depicted in column ``Step 4'' table~\ref{table:reduced_graphs}.
\end{itemize}

\begin{table}[t] 
\centering
\begin{tabular}{cccc}
Vertex & Step 1 & Step 3 & Step 4  \\
\hline \hline \\[-2mm]
\includefigure{width=0.4cm}{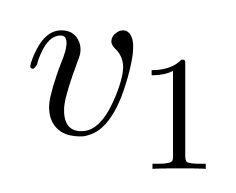}  &
\includefigure{width=3cm}{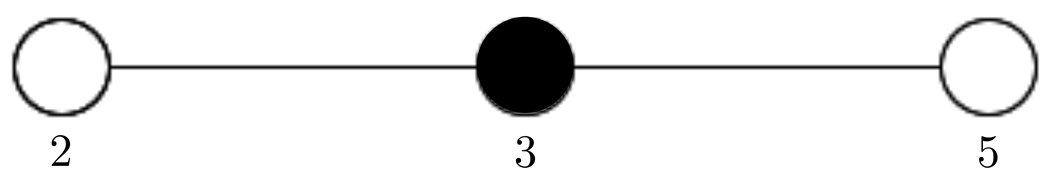} &
\includefigure{width=3cm}{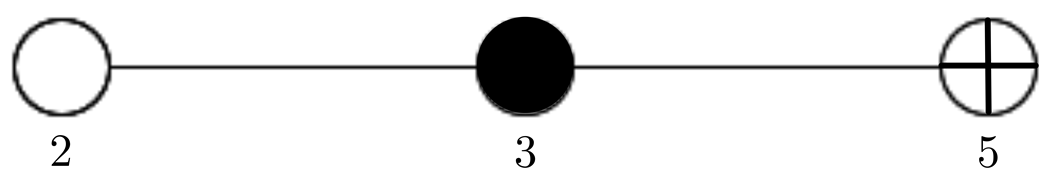} &
\includefigure{width=3cm}{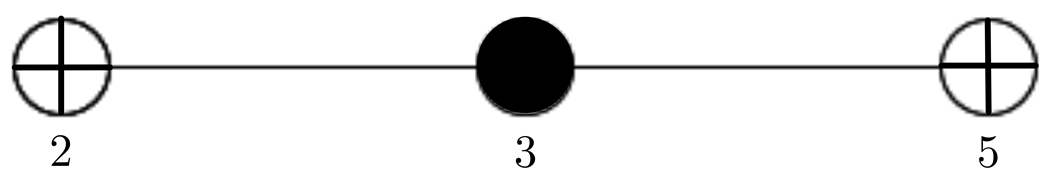} \\[5mm]

\includefigure{width=0.4cm}{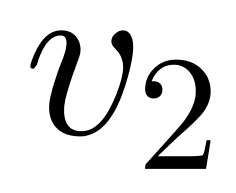}  &
\includefigure{width=3cm}{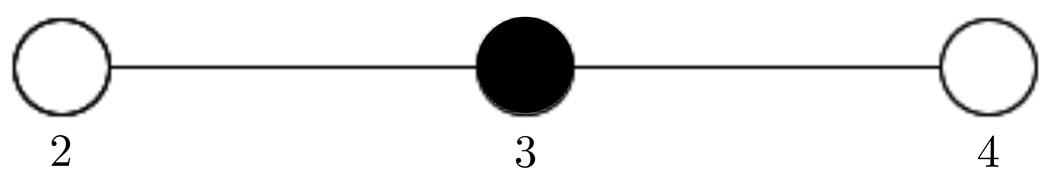} &
\includefigure{width=3cm}{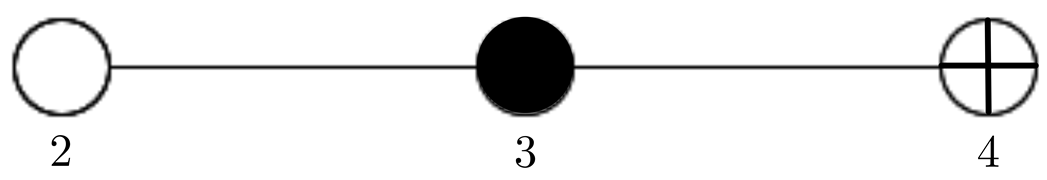} &
\includefigure{width=3cm}{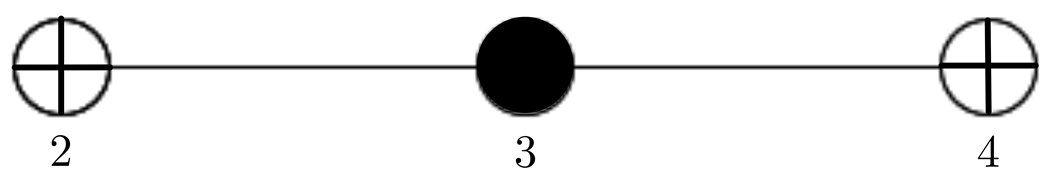} \\[5mm]

\includefigure{width=0.4cm}{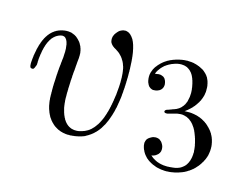}  &
\includefigure{width=3cm}{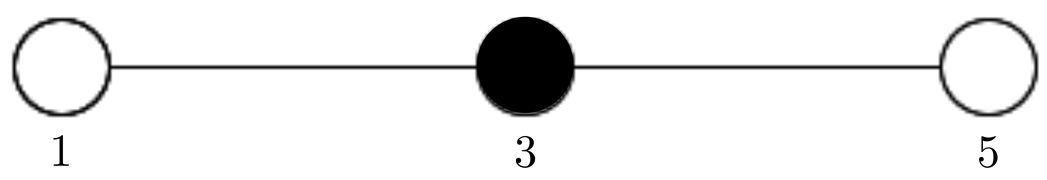} &
\includefigure{width=3cm}{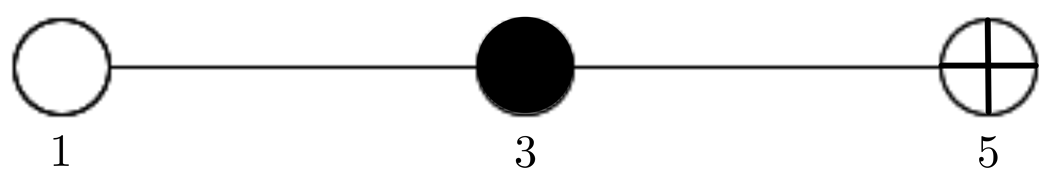} &
\includefigure{width=3cm}{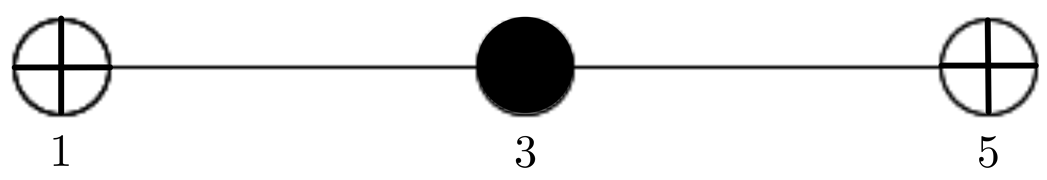} \\[5mm]

\includefigure{width=0.4cm}{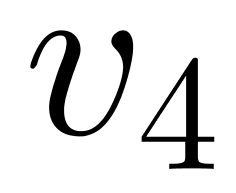}  &
\includefigure{width=3cm}{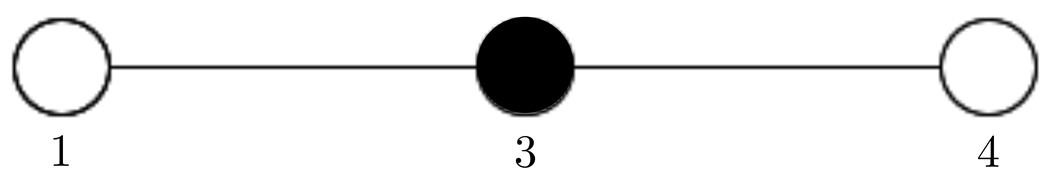} &
\includefigure{width=3cm}{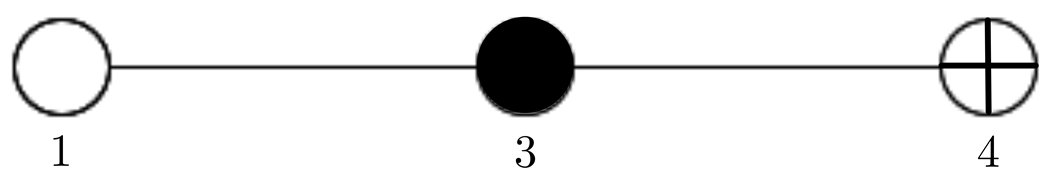} &
\includefigure{width=3cm}{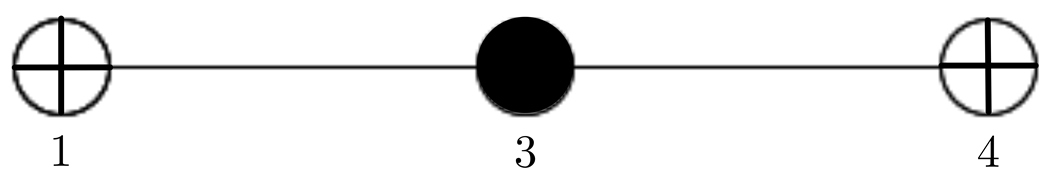} \\[5mm]

\includefigure{width=0.4cm}{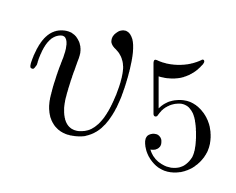}  &
\includefigure{width=2.3cm}{Graph_3,4} &
\includefigure{width=2.3cm}{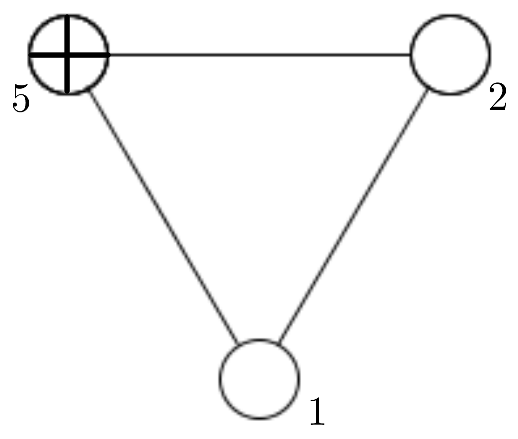} &
\includefigure{width=2.3cm}{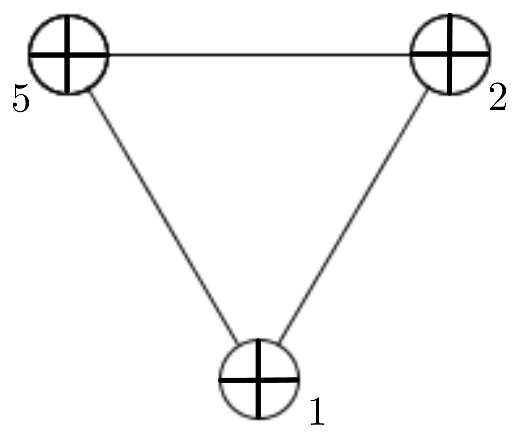} \\[5mm]

\includefigure{width=0.4cm}{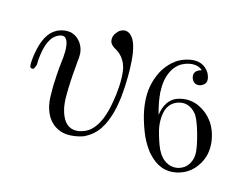}  &
\includefigure{width=2.3cm}{Graph_3,5} &
\includefigure{width=2.3cm}{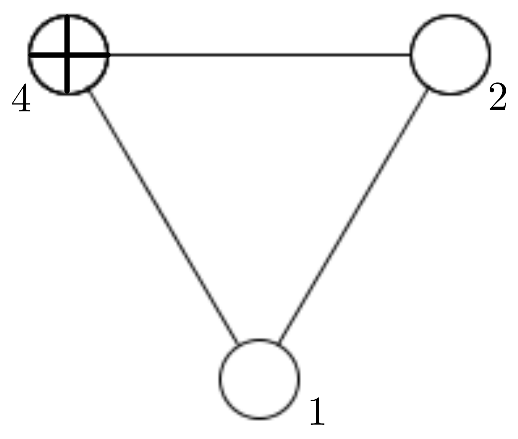} &
\includefigure{width=2.3cm}{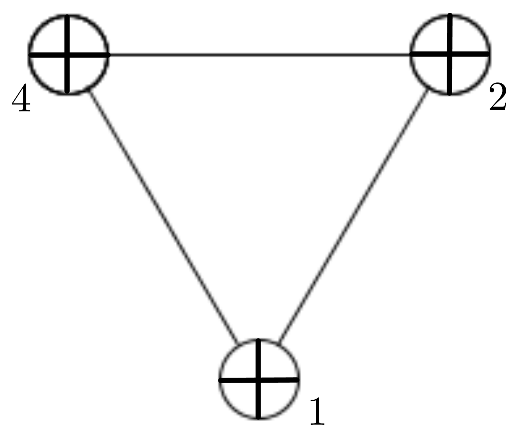} \\[5mm]

\hline \hline
\end{tabular}
\caption{\footnotesize{The graphs obtained in each step of the reduction algorithm
for $\Gcal$.}}
\label{table:reduced_graphs}
\end{table}

Since $\Gcal$ is admissible and
has an equilibrium $q\in \inter\left(\Gamma_{(3,2)}\right)$, by
Theorem~\ref{theor:Reduced_Poly_Conserv} we have that its limit dynamics on the attractor
$\attr{\Gcal}$ is described by a Hamiltonian polymatrix replicator in a lower dimensional prism.
Considering the strategy $3$ in group $1$, by Definition~\ref{nl,Al} we obtain the 
$(q,3)$-reduction $((2,2),A(3))$ where $\tilde{A}:=A(3)$ is the matrix
$$\tilde{A}=\left[\begin{array}{ccccc}
		    -9 & 9  & 9  & -9 \\
			-9 & 9  & 9  & -9 \\
			-6 & 6  & 6  & -6 \\
			-6 & 6  & 6  & -6 \\
			  \end{array}\right]\,.$$

Consider now the polymatrix replicator associated to the game
\linebreak $\tilde{\Gcal}=\left((2,2),\tilde{A}\right)$,
which is equivalent to the trivial game
$\left((2,2),0\right)$. Hence its replicator dynamics on the polytope
$\Gamma_{(2,2)}=\Delta^1\times\Delta^1$ is trivial, in the sense that all points are equilibria.
In particular the associated vector field 
 $X_{\tilde{\Gcal}}=0$ is Hamiltonian.
 
Since the reduced  information set $\Rscr(\Gcal)$ is of type
$\{\bullet,\oplus\}$, by Proposition~\ref{prop poly reduced} the flow of $X_\Gcal$ admits an invariant foliation with a single globally attractive equilibrium on each leaf (see Figure~\ref{two_leafs}).
Therefore, the attractor $\Lambda_{\Gcal}$ is just a line segment of equilibria, which embeds in the Hamitonian flow of 
 $X_{\tilde{\Gcal}}=0$, 
as asserted by Proposition~\ref{Reduced_Poly_1}.

\section*{Acknowledgments}

The first author was supported by IMPA through a  p\'os-doutorado de excel\^encia position. 
The second author was supported by  Funda\c{c}\~{a}o para a Ci\^{e}ncia e a Tecnologia,  UID/MAT/04561/2013.
The third author was supported by FCT scholarship  SFRH/BD/72755/2010.


\nocite{*}
\bibliographystyle{amsplain} 


\end{document}